\documentclass[reqno]{amsart}
\usepackage{amsthm,amstext}
\usepackage{zito}
\usepackage{mathptmx}       
\usepackage{helvet}         
\usepackage{courier}        
\usepackage{color}
\usepackage{soul}
\usepackage[charter]{mathdesign}
\makeatletter \@addtoreset{equation}{section} \makeatother

\renewcommand\thefigure{\thesection.\@arabic\c@figure}
\renewcommand\thetable{\thesection.\@arabic\c@table}
\newtheorem{theorem}{Theorem}[section]
\newtheorem{lemma}[theorem]{Lemma}
\newtheorem{proposition}[theorem]{Proposition}
\newtheorem{corollary}[theorem]{Corollary}
\newtheorem{conjecture}[theorem]{Conjecture}
\newtheorem{definition}[theorem]{Definition}

\theoremstyle{remark}
\newtheorem{remark}[theorem]{Remark}

\newcommand{\mc}[1]{{\mathcal #1}}

\newcommand{\bb}[1]{{\mathbb #1}}

\newcommand{\<}{\langle}
\renewcommand{\>}{\rangle}

\DeclareMathOperator{\supp}{supp} 
\DeclareMathOperator{\spanx}{span}

\DeclareMathOperator{\sgn}{sgn}


\title{Density fluctuations for exclusion processes with long jumps}
\date{}

\author{Patr\'icia Gon\c{c}alves}
\address{Patr\'icia Gon\c{c}alves, Center for Mathematical Analysis,  Geometry and Dynamical Systems \\
Instituto Superior T\'ecnico, Universidade de Lisboa \\
Av. Rovisco Pais,1,  1049-001 Lisboa, Portugal
\newline e-mail: \rm \texttt{patricia.goncalves@math.tecnico.ulisboa.pt}}

\author{Milton Jara}
\address{ Milton Jara\\ IMPA\\ Estrada Dona Castorina 110\\ Jardim Bot\^anico\\ CEP 22460-340\\ Rio de Janeiro\\ Brazil
 \newline e-mail: \rm \texttt{mjara@impa.br}}

\begin{document}

\begin{abstract}
We show that the stationary density fluctuations of exclusion processes with long jumps, whose rates are of the form $c^\pm |y-x|^{-(1+\alpha)}$ where $c\pm$ depends on the sign of $y-x$, are given by a fractional Ornstein-Uhlenbeck process for $\alpha \in (0,\frac{3}{2})$. When $\alpha =\frac{3}{2}$ we show that the density fluctuations are tight, in a suitable topology, and that any limit point is an energy solution of the fractional Burgers equation, previously introduced in \cite{GubJar} in the finite volume setting.
\end{abstract}

\maketitle

\section{Introduction}

A classical problem on the field of interacting particle systems corresponds to the derivation of a scaling limit for the stationary\footnote{
In the physics literature, a stationary state is what in probability is called an invariant measure, and an equilibrium state corresponds to an invariant measure which is in addition reversible.
}
fluctuations of the conserved quantities of the system. The archetypical example is the {\em exclusion process}, which we describe as follows. The exclusion process is a system of particles on a given graph, on which each particle performs a continuous-time random walk with the restriction that each site on the graph is allowed to have at most one particle. Despite its simplicity, the richness of this process makes of it one of the favorite models on the realm of interacting particle systems. In these notes, we consider the exclusion process with {\em long jumps} on the one-dimensional lattice, introduced in \cite{Jar0}\footnote{
The exclusion process with arbitrary, translation invariant transition rates is well understood (see \cite{Lig}). However, as far as we know, the first article where the particular properties of the exclusion process with long jumps were studied is \cite{Jar0}.
}. In this case, the transition rates of the underlying random walk have a polynomial tail of the form $c^{\pm} |y-x|^{-(1+\alpha)}$ for some $\alpha \in (0,2)$, where $c^{\pm}$ depends on the sign of $y-x$. The Bernoulli product measures $\mu_\rho$ of density $\rho \in [0,1]$ on $\{0,1\}^{\bb Z}$ are invariant under the evolution of this process, reflecting the fact that particles are neither created nor destroyed by the dynamics and the translation invariance of the transition rates. In these notes, we will study the stationary {\em density fluctuations} of the exclusion process with long jumps starting from $\mu_\rho$. For $\alpha \in (0,3/2)$, we show that the scaling limit of the density fluctuations are given by the infinite-dimensional Ornstein-Uhlenbeck equation
\begin{equation}
\label{ec0.1}
d \mc Y_t = (\mc L^\rho)^\ast \mc Y_t dt  + \sqrt{2\rho(1-\rho)(-\mc L^{1/2})} d{\mc W}_t,
\end{equation}
where ${\mc W}_t$ is a Brownian motion, $\mc L^{\rho}$ is the generator of an $\alpha$-stable skewed L\'evy process given in \eqref{operatorLrho}, $(\mc L^\rho)^*$ is its adjoint and $\mc L^{1/2}$ is the symmetric part of $\mc L^\rho$. In the case  $\alpha =3/2$ we prove that the density fluctuation field is tight and any limit point is an energy solution of the {\em fractional Burgers equation}
\[
d \mc Y_t = (\mc L^\rho)^\ast \mc Y_t dt + m \nabla \mc Y_t^2 dt +  \sqrt{2\rho(1-\rho)(-\mc L^{1/2})} d{\mc W}_t,
\]
where $m$ is the mean of the underlying transition rate. The notion of energy solution of the previous equation was introduced in \cite{GonJar1} in the context of the KPZ equation
\[
 dh_t = \Delta h_t dt+ (\nabla h_t)^2  dt+ d{\mc W}_t.
\]
The fractional Burgers equation was introduced in \cite{GubJar} in finite volume and the notion of energy solutions was used to prove existence of solutions of this equation.  Recently in \cite{GubPer}, uniqueness of such solutions has been proved in the stationary case. This result, combined with the results in \cite{GonJar1} complete the proof of the weak KPZ universality conjecture in the stationary case. However, the methods of \cite{GubPer} do not generalize to the fractional Burgers equation. In the past few years, a great deal of research around the KPZ/Burgers equation and its universality class has been done; see \cite{Cor} for a review. A fundamental breakthrough on the mathematical understanding of the wellposedness of the KPZ equation has been given in \cite{Hai1, Hai2}, settling on firm grounds questions about existence and uniqueness of solutions of the KPZ equation. At least heuristically, the theory of regularity structures gives uniqueness of solutions for the fractional Burgers equation in the regime $\alpha \in (3/2,2]$. However, the theory of regularity structures, in its current formulation, breaks down exactly at $\alpha =3/2$, which is the parameter where our fractional Burgers equation appears.\\

The main motivation for these notes comes from the {\em strong KPZ universality conjecture}, which, roughly speaking, states that there is a universal object (the KPZ fixed point) that governs the fluctuations of stationary, non-equilibrium, conservative, one-dimensional stochastic models. Starting from various physical considerations, one important property of this universal object is its scale invariance with respect to the KPZ space-time scaling $1:2:3$. The fractional Burgers equation is invariant under this scaling, and therefore it provides a candidate for, at least, the equation satisfied by the KPZ fixed point. As far as we know, this is the first example of a non-linear equation with a meaningful notion of solution, obtained as a scaling limit of a stochastic, conservative system, which is invariant under the KPZ scaling.\footnote{
The fractional Burgers equation considered in \cite{GubJar} is defined in finite volume, and therefore the spatial scaling changes its domain.
}
In \cite{CorQuaRem}, the authors propose another candidate for the KPZ fixed point. However, it seems that even the existence of the object defined in \cite{CorQuaRem} is not proved rigorously. \\

Our method of proof is an improvement over the proof carried out in \cite{GonJar1}, where the finite-range case is treated. The main technical novelty is the treatment of the non-local part of the drift, which requires a multiscale analysis which is different from the one introduced in \cite{GonJar1} and similar to the one introduced in \cite{Gon}. The idea taken from \cite{GonJar1} is the following. Consider for simplicity a local observable of the dynamics, it could be, for example, the occupation number at the origin. Due to the conservation of the number of particles and the ergodicity of the dynamics, the local density of particles is the observable of the dynamics which takes more time to equilibrate. Therefore, if we look at the evolution on the right space-time scale, any observable of the dynamics should be asymptotically equivalent to a function of the density of particles on a block of, a macroscopical, small size around the support of the observable. The Boltzmann-Gibbs principle introduced in \cite{BroRos} states that, at first order, this function is {\em linear} on the density of particles; a claim supported by the equivalence of ensembles. The second-order Boltzmann-Gibbs principle introduced in \cite{GonJar1}, states that the second-order correction term is a {\em quadratic} function of the local density of particles. This allows to replace any {\em local} function of the dynamics by the corresponding function of the local density of particles. In these notes, the drift is a non-local function, and the multiscale analysis introduced in \cite{GonJar1} is not enough to handle this non-local function, so we introduce a second multiscale which, combined with the original one, allows to replace the drift by a quadratic function of the local density of particles. For $\alpha \leq 1$, this sophisticated method is not needed and the fluctuations can be obtained by means of classical methods. For $\alpha \in (1,1+ \frac{2}{5+\sqrt{ 33}})$, one step of the multiscale analysis of \cite{GonJar1} is needed, so the proof is not very different from the case $\alpha \leq 1$. For $\alpha \in [1+ \frac{2}{5+\sqrt{ 33}},\frac{3}{2})$, the multiscale analysis shows that the drift term vanishes in the limit, which is the reason why the Ornstein-Uhlenbeck equation \eqref{ec0.1} is the limit in those cases. The division between the cases $\alpha \in (1,1+\frac{2}{5+\sqrt{33}})$ and $\alpha \in [1+\frac{2}{5+\sqrt{33}},\frac{3}{2})$ is rather artificial, and it is done just to emphasize that in order to obtain our result in full generality, it is necessary to introduce new ideas, which come, in these notes, in the form of a refined multiscale analysis.  For $\alpha =\frac{3}{2}$ the drift makes its way up to the limit, in the form of a quadratic functional of the limiting field. This quadratic functional is extremely singular, and it is the source of trouble for the stochastic Burgers equation. Only after \cite{Hai1} we have been able to understand how to set up correctly a well-posed Cauchy problem for the (local) stochastic Burgers equation. The theory of regularity structures works thanks to the following heuristic observation: the scaling of the nonlinearity of the equation is  subcritical, with respect to the scaling of the linear part. Therefore, the theory of regularity structures makes possible to set up a Picard iteration scheme to solve it. This observation is no longer true for the fractional Burgers equation: the equation is {\em critical}, in the sense that, the nonlinear part and the linear part scale in the same way. Therefore, it is not surprising that we are not able to obtain a full convergence result for $\alpha =\frac{3}{2}$. Notice, however, that we have enough information about limit points to show that they are well defined as stochastic processes, that the nonlinearity is well defined in a strong sense, and that they solve a martingale formulation of the fractional Burgers equation.\\

These notes are organized as follows. In Section \ref{s1} we define the exclusion process with long jumps and we make precise formulations of the main results of the article. These formulations require a great deal of previous definitions, which are carried out along the section. In particular we need to define what do we understand by stationary solutions of the fractional Ornstein-Uhlenbeck equation and by stationary energy solutions of the fractional Burgers equation. A great deal of care is needed at this point. It is natural to consider the density fluctuation field as a distribution-valued stochastic process. Therefore, its action is well defined for test functions in the Schwartz space $\bb S(\bb R)$. But $\bb S(\bb R)$ is not left invariant by $\mc L^\rho$. Recall that the operator $\mc L^\rho$ is given in \eqref{operatorLrho}. In fact, for most functions $f \in \bb S(\bb R)$, $\mc L^\rho f $ does not belong to $\bb S(\bb R)$.  This fact is easy to verify in Fourier space. But stationary processes are stochastically continuous in $L^2(\bb R)$, which allows to define their action over functions in $L^2(\bb R)$ through suitable approximations.\\

The general strategy of proof of the main results of these notes is not difficult to describe. Our definitions of solutions use martingale characterizations. Therefore, we will verify that the density fluctuations satisfy an  approximate martingale problem, which, in the limit, becomes the martingale problem associated to the corresponding limiting processes. The passage to the limit is allowed by tightness arguments, complemented by some uniform estimates on the errors of the approximation.\\

%

 In Section \ref{s2} we define and compute various martingales associated to the density fluctuation fields, which will be used to show that the density fluctuation fields satisfy an approximated version of the martingale problem defined for the limiting processes.  \\
 

 In  Sections \ref{sec:tight_ini_field}, \ref{sec:tight_mart} and \ref{sec:tight_integral} we prove tightness of some terms in the martingale decomposition of the density fluctuation field, which work out for any $\alpha\in(0,2)$.  The crucial part is to deal with the drift term $A_t^n(f)$. In Section \ref{s3.1} we prove  tightness  of $A_t^n(f)$ in the case $\alpha \leq 1$. This case is a good warm-up to what follows next, since the standard proof found, for example, in Chapter 11 of \cite{KipLan} works well.\\
   
     In Section \ref{s3.2} we prove tightness of $A_t^n(f)$ in the case $1<\alpha < 1+\frac{2}{5+\sqrt{33}}$.   In this section  we also state an estimate on the variance of additive functionals of the processes of Kipnis-Varadhan's type and we use the spectral gap inequality, stated in Appendix \ref{ap:A}, in order to transform it into an effective estimate, stated as Proposition \ref{p2.3.4}. We point out that once we have established Proposition \ref{p2.3.4}, the text is completely independent of the Kipnis-Varadhan's inequality or of the spectral gap inequality. In particular, if, by some other means, we were able to prove Proposition \ref{p2.3.4}, the results of these notes follow without needing these inequalities. We also  state the form of the equivalence of ensembles that will be needed in the proofs  and we  show how to use Proposition \ref{p2.3.4} in order to get estimates on the variance of the drift term. We only have at our disposal a brute-force Cauchy-Schwarz estimate to deal with the tail part of the drift term. In this way we can show the asymptotic negligibility of jumps of size bigger than $n^{\frac{2\alpha-2}{2\alpha-1}}$, see Lemma \ref{lem_1}. This jump size corresponds to   macroscopical small jumps. The smaller jumps can be handled with Proposition \ref{p2.3.4}. A single use of this proposition is enough to fill the gap up to $n^{\frac{2\alpha-2}{2\alpha-1}}$, only for $\alpha < 1+\frac{2}{5+\sqrt{33}}$, so a more refined argument is needed for the general case. \\
     
     In Section \ref{s3.3} we prove tightness for $1+\frac{2}{5+\sqrt{33}}\leq \alpha <\frac{3}{2}$. In this case we need to introduce a multiscale analysis in order to use in an effective way Proposition \ref{p2.3.4}. The idea is the following. Proposition \ref{p2.3.4} allows to estimate the variance of space-time additive functionals of the dynamics by their spatial variance, paying as a price the inverse of the spectral gap over the support of the spatial functions in consideration. Therefore, the largest the support of the functions we consider, the less effective Proposition \ref{p2.3.4} is. The current associated to big jumps has a very big support, but its variance decays with the distance as a power law. Therefore, there is a trade-off between the support and the intensity of a big jump. The right way to exploit this trade-off is through a multiscale analysis.\\

In Section \ref{s3.4} we prove tightness for $\alpha =\frac{3}{2}$. Although the multiscale analysis of Section \ref{s3.3} still makes big jumps negligible, very small jumps are no longer negligible and a new argument is needed. The multiscale analysis of Section \ref{s3.3} stops at size $n^{1-\delta}$ for some small $\delta>0$ and it shows that the drift term is asymptotically equivalent to the square of the density on a box of size $n^{1-\delta}$. This is what is called the {\em one-block estimate} in the literature of interacting particle systems. Using the renormalization scheme introduced in \cite{GonJar1}, we show the {\em two-blocks estimate}, which states that the drift term is asymptotically equivalent to the square of the density on a box of size $\varepsilon n$. This method shows a uniform $L^2$-bound for the difference between the drift term and the square of the density, which is good enough to prove tightness by means of the Kolmogorov-Centsov's criterion stated in Proposition \ref{KCC}.\\

In Section \ref{s3.5} we show  Theorem \ref{t1.5.1}, that is  the convergence of the density fluctuation field to the stationary solution of the fractional Ornstein-Uhlenbeck equation. Once tightness is proved, the proof is standard and  relies on the martingale characterization of such solutions.  In Section \ref{s3.6} we show Theorem \ref{t1.6.2}, which is also not very difficult to prove once the uniform bound \eqref{ec3.6.1} is obtained.\\

Section \ref{s6.1} contains a discussion about how the main result of this article is related to the so-called {\em KPZ universality} conjecture. It turns out that energy solutions of \eqref{SBE} are invariant under the KPZ scaling exponents $1:2:3$. It can be checked that, at a formal level, modulo space-time rescalings, stationary solutions of equation \eqref{SBE} form a two-parameter family of self-similar processes, each of them invariant under the KPZ scaling. 
It has been conjectured (see \cite{CorQuaRem} for example) that fluctuations of one-dimensional growth interfaces converge to a universal process, dubbed the {\em KPZ fixed point}. This object should also be invariant with respect to the KPZ scaling. 
Another universality class that also has the KPZ scaling exponents is the scaling limit of energy fluctuations of one-dimensional stochastic models of heat conduction in the so-called {\em zero pressure point}. In that case, the limiting process is the Ornstein-Uhlenbeck process associated to the fractional Laplacian $\mc L^\rho$, which corresponds to \eqref{ec1.5.3} with $m=0$ see \cite{BerGonJar,JarKomOll}. The stability index $\alpha=3/2$ appears naturally in the solution of a two-dimensional Poisson equation involving local operators. The skewness parameter depends on the details of the model and it can have any admissible value. We conjecture that the family of solutions of \eqref{ec1.5.3} converges, as $m \to \infty$ and after a proper time scaling, to the KPZ fixed point. We say in that case that \eqref{ec1.5.3} {\em interpolates} between the Gaussian case $m=0$ and the KPZ fixed point. In \cite{AmiCorQua}, it is proved that the KPZ equation can be understood a {\em crossover} between the Ornstein-Uhlenbeck process associated to the usual Laplacian and the KPZ fixed point. These two universality classes have different scaling exponents, and in particular the KPZ equation is not scale invariant. Therefore, our conjecture is conceptually different, since all the interpolating processes have the same scaling exponents.\\

Section \ref{s6.2} explains how to obtain the fractional Burgers equation associated to the fractional Laplacian $-(-\Delta)^{\alpha/2}$ as a scaling limit of a family of asymmetric exclusion processes. Since the limiting equation is not scale invariant for $\alpha  \neq 3/2$, one has to tune the strength of the asymmetry with respect to the scaling. For $\alpha > 3/2$, the strength of the asymmetry decreases with $n$ (we say that the system is {\em weakly asymmetric}), and for $\alpha < 3/2$ the strength of the asymmetry increases with $n$.\\

Section \ref{s6.3} explains how to extend the results of this paper to the case on which the jump rates belong to the domain of {em normal attraction} of an $\alpha$-stable law. \\

Finally, Section \ref{s6.4} gives some examples of particle systems with long-range interactions for which the methods of this article allow to prove the results stated in this article.

\section{Notations and results}
\label{s1}
\subsection{The exclusion process}

In this section we describe what it is known in the literature as the {\em exclusion process} on the one-dimensional lattice $\bb Z$. We say that a function $p: \bb Z \to [0,\infty)$ is a {\em transition rate} if $p(0)=0$ and $p^* = \sum_{z} p(z) < +\infty$. From now on, if the set of indices of a sum is not specified, we assume that it is equal to $\bb Z$.
Let $p(\cdot)$ be a transition rate. Let $\Omega = \{0,1\}^{\bb Z}$ be the state space of a Markov process. We consider on $\Omega$ the product topology. We denote by $\eta = \{\eta(x)\,;\, x \in \bb Z\}
$ the elements of $\Omega$. We say that $x \in \bb Z$ is a {\em site} and that $\eta \in \Omega$ is a {\em configuration of particles}.
Let $\eta \in \Omega$ be a configuration of particles. We say that there is a particle at the site $x$ if $\eta(x) = 1$; otherwise we say that the site $x$ is {\em empty}. Let us consider the following dynamics. Each particle waits an exponential time of rate $p^*$ (for this reason we ask $p^*$ to be finite), at the end of which it chooses a site $y \in \bb Z$ with probability $p(y-x)/p^*$, where $x$ is the current position of the particle. If the chosen site is empty, the particle jumps into it. Otherwise it stays at its current position. In any case, a new exponential time starts afresh and the particle repeats the steps above.

The dynamics described above corresponds to a Markov process $\{\eta_t\,;\, t \geq 0\}$ defined on $\Omega$. If the number of particles is finite, it is not difficult to construct $\eta_t$ for any $t \geq 0$. When the number of particles is infinite, a detailed construction of the process $\{\eta_t\,;\, t \geq 0\}$ can be found in \cite{Lig}. In particular, the derivation of the properties we will describe below can be found there. Notice that we are not assuming anything about $p(\cdot)$ aside from $p^* < +\infty$. In particular, we can assume that particles perform arbitrarily long jumps with positive probability, which is the case we are interested in on these notes.

We say that a function $f: \Omega \to \bb R$ is {\em local} if there exists $A \subseteq \bb Z$ finite such that $f(\eta) = f(\xi)$ whenever $\eta(x) = \xi(x)$ for any $x \in A$. We say that the smallest of such sets is the {\em support} of $f$, and we denote it by $\supp(f)$. For $\eta \in \Omega$ and $x,y \in \bb Z$ we define $\eta^{x,y} \in \Omega$ as
\begin{equation}
\eta^{x,y}(z) =
\begin{cases}
\eta(y); & z=x\\
\eta(x); & z = y\\
\eta(z); & z \neq x,y.
\end{cases}
\end{equation}

For $f: \Omega \to \bb R$ and $x,y \in \bb Z$ we define $\nabla_{x,y} f : \Omega \to \bb R$ as $\nabla_{x,y} f(\eta) = f(\eta^{x,y})-f(\eta)$ for any $\eta \in \Omega$. For a local function $f: \Omega \to \bb R$ we define $Lf : \Omega \to \bb R$ as
\[
L f(\eta) = \sum_{x,y} p(y-x) \eta(x) \big(1-\eta(y)\big) \nabla_{x,y} f(\eta) \text{ for any } \eta \in \Omega.
\]

Since $f$ is a local function, we note that the sum above has a finite number of non-zero entries, and, in particular, $L f$ is well defined. The linear operator $L$ defined in this way turns out to be closable with respect to the uniform topology on the space $\mc C(\Omega)$ of continuous functions $f: \Omega \to \bb R$. Moreover, its closure (also denoted by $L$) turns out to be the generator of the process $\{\eta_t\,;\, t \geq 0\}$.

For $\rho \in [0,1]$, let $\mu_\rho$ be the probability measure in $\Omega$ given by
\[
\mu_\rho\big\{\eta\in\Omega: \eta(x_1)=1,\dots, \eta(x_\ell) =1\big\} = \rho^ \ell
\]
for any finite collection $\{x_1,\dots,x_\ell\}$ of sites in $\bb Z$. The measures $\{\mu_\rho\,;\, \rho \in [0,1]\}$ are invariant under the evolution of $\{\eta_t\,;\, t \geq 0\}$. If $\spanx\{x \in \bb Z\,;\, p(x) >0\} = \bb Z$ (that is, if $p(\cdot)$ is {\em irreducible}) the measures $\{\mu_\rho\,;\, \rho \in [0,1]\}$ are also {\em ergodic} under the evolution of $\{\eta_t\,;\, t \geq 0\}$.

\subsection{Random walks with long jumps}
\label{s1.2}

Let $\alpha \in (0,2)$ and $c^+, c^- \geq 0$ be such that $c^++c^->0$. Let us define $p: \bb Z \to [0,\infty)$ as
\begin{equation}
\label{ec1}
p(z) = \frac{c(z)}{|z|^{1+\alpha}}, \text{ where } c(z) =
\begin{cases}
c^+; & z>0,\\
0; & z=0,\\
c^-; &z <0.
\end{cases}
\end{equation}
Notice that the condition $\alpha >0$ ensures that $p(\cdot)$ is a transition rate. Let $\{x(t)\,;\, t \geq 0\}$ be the continuous-time random walk on $\bb Z$ with transition rate $p(\cdot)$. The following is a classical result which can be found, for instance, in Chapter 1, Theorem 2.4 of \cite{BorIbra}.

\begin{proposition}
\label{p1.1.2}
The process $\{x_t^n\,;\, t \geq 0\}$ given by
\begin{equation}\label{mnalpha}
x_t^n = \tfrac{1}{n} \Big( x(tn^\alpha)- m_n^\alpha t\Big), \text{ where } m_n^\alpha=
\begin{cases}
0; & \alpha <1\\
n\sum_{|x| \leq n} xp(x); & \alpha =1\\
n^\alpha \sum_x xp(x); & \alpha >1
\end{cases}
\end{equation}
converges in distribution, as $n\to\infty$, to a Markov process $\{Z_t\, ;\, t \geq 0\}$.
\end{proposition}

The generator $\mc L = \mc L(c^+,c^-;\alpha)$ of the process $\{Z_t\, ;\, t \geq 0\}$ is given by
\begin{equation}
\label{ec2}
\mc L f(x) = \int\limits_{\bb R} \frac{c(y)}{|y|^{1+\alpha}} \Big( f(x+y) - f(x) - \theta^\alpha(y) f'(x) \Big) dy,
\end{equation}
where
\begin{equation}\label{eq2}
\theta^\alpha(y) =
\begin{cases}
0; & \alpha <1\\
y \mathbf{1}(|y| \leq 1); & \alpha =1\\
y; & \alpha >1.
\end{cases}
\end{equation}
Note that the generator of the process $\{x_t^n\,;\, t \geq 0\}$ is given by
\begin{equation}\label{ec'2}
\mc L_n f\big( \tfrac{x}{n} \big) = n^\alpha \sum_{y} p(y) \Big( f\big(\tfrac{x+y}{n}\big) - f\big( \tfrac{x}{n}\big) \Big) - \tfrac{m_n^\alpha}{n} f'\big(\tfrac{x}{n}\big).
\end{equation}
The generator acts on functions $f: \bb R \to \bb R$. We have denoted real numbers as $\frac{x}{n}$ to emphasize that, aside from a constant drift, the operator $\mc L_n$ is discrete in nature.
When $c^+=c^-$, that is, when the transition rate $p(\cdot)$ is {\em symmetric}, the operator $\mc L$ is a constant multiple of the {\em fractional Laplacian} $-(-\Delta)^{\alpha/2}$.

Another, more analytical, way to face a result like the one in Proposition \ref{p1.1.2} is through the convergence of generators. In fact, we have the following result.

\begin{proposition}
\label{p2}
Let $f: \bb R \to \bb R$ be a function in $\mc C^2(\bb R)$. Then,
\begin{equation}\label{resprop2.2}
\begin{split}
&\lim_{n \to \infty} \sup_{x } \Big|\mc L_n f \big( \tfrac{x}{n} \big) - \mc L f \big( \tfrac{x}{n} \big) \Big| =0,\\
&\lim_{n \to \infty} \tfrac{1}{n} \sum_{x} \Big| \mc L_n f\big(\tfrac{x}{n}\big) - \mc L f \big( \tfrac{x}{n} \big) \Big| =0.
\end{split}
\end{equation}
\end{proposition}

This result is classical and a proof may be found on Appendix \ref{proofprogen}. For $x\in\bb Z$, let us define the symmetric part $s(x)$ and the antisymmetric part $a(x)$  of $p(x)$, respectively,  as
\[
s(x)= \tfrac{1}{2} \big( p(x) + p(-x)\big), \quad a(x) =\tfrac{1}{2} \big( p(x)-p(-x)\big).
\]
 Let us note that
\[
s(x) = \frac{c^++c^-}{2|x|^{1+\alpha}}, \quad a(x) = \frac{c^+-c^-}{2|x|^{1+\alpha}}\sgn{(x)}.
\]
An important functional associated to the operator $\mc L$ is the {\em Dirichlet form} defined as
\begin{equation}\label{cont dir form}
\mc E(f) = - \int\limits_{\bb R} f(x) \mc L f(x) dx = \frac{c^++c^-}{4}\iint\limits_{\bb R^2} \frac{(f(y)-f(x))^2}{|y-x|^{1+\alpha}} dx dy.
\end{equation}
The discrete counterpart of this functional is given by
\begin{equation}\label{disc dir form}
\mc E_n(f) = \frac{n^{\alpha-1}}{2} \sum_{x,y} s(y-x) \big( f\big(\tfrac{y}{n}\big) -f\big(\tfrac{x}{n}\big)\big)^2,
\end{equation}
since the previous sum is the Riemann sum of the integral $\mc E(f)$, because it can be rewritten as 
\[
\mc E_n(f) = \frac{c^++c^-}{4n^2} \sum_{x,y}  \frac{n^{1+\alpha}}{|y-x|^{1+\alpha}}{ \big( f\big(\tfrac{y}{n}\big) -f\big(\tfrac{x}{n}\big)\big)^2}.
\]

We have the following

\begin{proposition}
\label{p3}
Let $f: \bb R \to \bb R$ be a smooth function. Then,
\[
\lim_{n \to \infty} \mc E_n(f) = \mc E(f).
\]
\end{proposition}
This proposition is a simple consequence of Proposition \ref{p2} or simply by noticing the limit of the Riemann sum to the double integral.

\subsection{The Ornstein-Uhlenbeck process}
Let $\mc C_c^\infty(\bb R)$ be the set of infinitely differentiable functions $f: \bb R \to \bb R$ with compact support. For $f \in \mc C_c^\infty(\bb R)$ and $k,\ell \in \bb N_0$\footnote{
We write $\bb N = \{1,2,...\}$ and $\bb N_0 = \{0\}\bigcup\bb N$.
}
we define
\[
\big\| f\big\|_{k,\ell,\infty} = \sup_{x \in \bb R} \big(1+x^2\big)^{k/2} \big| f^{(\ell)}(x)\big|.
\]
The Schwartz space of test functions is defined as the closure of $\mc C_c^\infty(\bb R)$ with respect to the metric
\[
d(f,g)  = \sum_{k,\ell} \frac{1}{2^{k+\ell}} \min \big\{ 1, \|f-g\|_{k,\ell,\infty}\big\}.
\]
This space is denoted by $\mathbb S(\bb R)$ and it coincides with the space of infinitely differentiable functions $f:\bb R \to \bb R$ such that $\|f\|_{k,\ell,\infty} <+\infty$ for any $k,\ell \in \bb N_0$. The space $\mathbb S'(\bb R)$ of {\em tempered distributions} is defined as the topological dual of $\mathbb S(\bb R)$. We will consider in $\mathbb S'(\bb R)$ the weak-$\star$ topology. We denote by
\[
\|f\| = \Big(\int_{\bb R} f(x)^2 dx \Big)^{1/2}
\]
the $L^2(\bb R)$-norm of $f$ and we denote by $\<f,g\> = \int f(x)g(x) dx$ the inner product between $f$ and $g$ in $L^2(\bb R)$.

One of the simplest examples of $\mathbb S'(\bb R)$-valued random variables is the so-called {\em white noise}. We say that an $\mathbb S'(\bb R)$-valued random variable $\omega$ is a white noise of variance $\chi$ if for any $f \in \mathbb S(\bb R)$ the real-valued random variable $\omega(f)$ has a Gaussian distribution of mean zero and variance $\chi \|f\|^2$.

Let $T >0$ be a fixed number. This number $T$ will be fixed up to the end of these notes. For a given topological space $E$ we denote by $\mc C([0,T]; E)$ the space of continuous functions from $[0,T]$ to $E$ and by $\mc D([0,T]; E)$ the space of {\em c\`adl\`ag} trajectories from $[0,T]$ to $E$.

We say that an $\mathbb S'(\bb R)$-valued process $\{\mc W_t; t \in [0,T]\}$ is a (standard) {\em Brownian motion} if for any function $f \in \mathbb S(\bb R)$ the real-valued process $\{\mc W_t(f); t \in [0,T]\}$ is a Brownian motion of variance $\|f\|^2$. For more details about the construction of this and other distribution-valued processes we refer to \cite{Wal}.

Let us recall the definition of the operator $\mc L$ given in \eqref{ec2}. In order to highlight the dependence of $\mc L$ on $c(\cdot)$ and $\alpha$, we use the notation   $\mathcal L(c^+,c^-;\alpha)$ for the choice of $c(\cdot)$ given in \eqref{ec1}. We will use the notation $\mc L(c^+,c^-;\alpha)$ whenever we need to stress this dependence. Let $\mc L^*$ be the adjoint of $\mc L$ in $L^2(\bb R)$ and let $$\mc S= \frac{1}{2}(\mc L+\mc L^*)$$ be the symmetric part of $\mc L$. We note that
\[
\mc L^*=\mc L(c^-,c^+;\alpha) \text{ and } \mc S = \mc  L(\tfrac{c^++c^-}{2}, \tfrac{c^++c^-}{2}; \alpha).
\]
Note that above, when we write $\mc L(c^-,c^+;\alpha)$ this corresponds to \eqref{ec2} but taking $c(\cdot)$ in \eqref{ec1} with $c^+$ exchanged with $c^-$.

Now we want to define what we understand as a {\em stationary solution} of the infinite-dimensional Ornstein-Uhlenbeck equation
\begin{equation}
\label{ec1.4.1}
d \mc Y_t = \mc L^* \mc Y_t dt + \sqrt{2\chi (-\mc S)} d \mc W_t,
\end{equation}
where $\{\mc W_t; t \in [0,T]\}$ is an $\mathbb S'(\bb R)$-valued Brownian motion and $\chi>0$ is fixed. A first na\"ive definition could be the following. We say that an $\mathbb S'(\bb R)$-valued process $\{\mc Y_t; t \in [0,T]\}$ is a solution of the {\em martingale problem} associated to \eqref{ec1.4.1} if for any differentiable trajectory $f:[0,T] \to \mathbb S(\bb R)$ the process
\begin{equation}\label{mart_OU}
\mc Y_t(f_t) - \mc Y_0(f_0) -\int_0^t \mc Y_s \big( (\partial_s+\mc L)f_s\big) ds
\end{equation}
is a martingale of quadratic variation
\[
2 \chi \int_0^t \<f_s,-\mc Sf_s\> ds
\]
with respect to the natural filtration associated to $\{\mc Y_t; t \in [0,T]\}$. This formulation has a serious problem: for general test functions $f$, $\mc Lf$ does not belong to $\mathbb S(\bb R)$ and therefore $\mc Y_t(\mc Lf)$ is not defined. The solution passes through the following property:

\begin{proposition}
\label{p1.4.1} The operator $\mc L : \mathbb S(\bb R) \to L^2(\bb R)$ is continuous. Moreover,  for any $f \in \mathbb S(\bb R)$, $\mc L f$ is bounded and infinitely differentiable.
\end{proposition}

Note that, by the definition of $\mc L$ given in \eqref{ec2}, for any $f \in \bb S(\bb R)$, $\mc L f' = (\mc Lf)'$ and in particular, from the previous proposition, $(\mc Lf)^{(\ell)}$ is bounded for any $\ell \in \bb N_0$.

\begin{definition}
We say that an $\mathbb S'(\bb R)$-valued process $\{\mc Y_t; t \in [0,T]\}$ defined on some probability space $(X,\mc F,P)$ is {\em stationary}\footnote{
This property assures us that the application $f \mapsto \mc Y_t(f)$ from $\mathbb S(\bb R) \subseteq L^2(\bb R)$ into $L^2(\Omega)$ is uniformly continuous, and equicontinuous on $t$. Any other property that would ensure this equicontinuity  would serve as a substitute to stationarity; nonetheless stationarity will be a consequence of other hypotheses needed to prove our main results, so we will not give too much attention to this point.
}
if for any $t \in [0,T]$ the $\mathbb S'(\bb R)$-valued random variable $\mc Y_t$ is a white noise of variance $\chi$.
 \end{definition}

 The constant $\chi$ above will be the same appearing in \eqref{ec1.4.1}. An important property of a stationary process is that $\mc Y_t(f)$ can be extended, by continuity, to any $f \in L^2(\bb R)$. In particular, for any $f \in \mathbb S(\bb R)$, the random variable $\mc Y_t( \mc L f)$ makes sense by Proposition \ref{p1.4.1}. In a more precise way, let $\psi \in \mathbb S(\bb R)$ be given by $\psi(x) = e^{-x^2}$ for any $x \in \bb R$ and define $\psi_\varepsilon \in \mathbb S(\bb R)$ as $\psi_\varepsilon(x) = \psi(\varepsilon x)$ for any $x \in \bb R$. Then $\psi_\varepsilon \mc L f \in \mathbb S(\bb R)$ for any $\varepsilon >0$ and any $f \in \mathbb S(\bb R)$. Moreover, $\psi_\varepsilon \mc L f \to \mc L f$ in $L^2(\bb R)$, as $\varepsilon \to 0$, from where we conclude that $\mc Y_t(\psi_\varepsilon \mc L f)$ converges in $L^2(P)$, as $\varepsilon\to0$, to a random variable which we call $\mc Y_t(\mc L f)$. \footnote{
Note that  $\mc Y_t(\mc Lf)$ is well defined up to a set of null probability. This set depends on the choice of the function $f$ and therefore we can not {\em a priori} think about $\mc Y_t(\mc L f)$ as a distribution-valued random variable. Nevertheless, stationarity and Proposition \ref{p1.4.1} can be used to show that $\mc Y_t(\mc L f)$ is indeed a distribution-valued random variable.
}

In order to give rigorous meaning to the Ornstein-Uhlenbeck equation \eqref{ec1.4.1} in a proper sense, we need to define the following object:
\begin{lemma}
\label{l1.4.2}
Let $\{\mc Y_t; t \in [0,T]\}$ be a stationary process. Let $f: [0,T] \to \bb S(\bb R)$ be differentiable. Then the process $\{\mc I_t(f); t \in [0,T]\}$ given by
\[
\mc I_t(f) = \lim_{\varepsilon \to 0} \int_0^t \mc Y_s(\psi_\varepsilon \mc L f_s) ds
\]
is well defined.
\end{lemma}
\begin{proof}
It is enough to observe that $\psi_\varepsilon \mc L f_s \to \mc L f_s$, as $\varepsilon \to 0$, in $L^2(\bb R)$, uniformly in $s$ and to note that $\mc Y_s$ is a linear functional  and a white noise with variance $\chi$.
\end{proof}

The previous lemma explains how to define the integral term in \eqref{mart_OU}, that is, the integral term on the martingale problem associated to the equation \eqref{ec1.4.1}. Let $\{\mc Y_t\,;\, t \in [0,T]\}$ be a stationary process and let $f: [0,T] \to \bb S(\bb R)$ be differentiable. We define
\[
\int_0^t \mc Y_s\big( \mc L f_s \big) ds =  \mc I_t(f).
\]
We say that a stationary process $\{\mc Y_t\,;\, t \in [0,T]\}$ is a {\em stationary solution} of equation \eqref{ec1.4.1} if for any differentiable function $f: [0,T] \to \bb S(\bb R)$ the process
\[
\mc Y_t(f_t) - \mc Y_0(f_0) - \int_0^t \mc Y_s\big( (\partial_s+\mc L) f_s\big) ds
\]
is a continuous martingale of quadratic variation
\[
2 \chi \int_0^t \<f_s, - \mc S f_s\> ds.
\]
Note that $\<f_s, - \mc S f_s\> = \mc E(f_s)$, where $\mc E(\cdot)$ is defined in \eqref{cont dir form}. The following proposition explains in which sense the stationary solutions of \eqref{ec1.4.1} are unique.

\begin{proposition}
\label{p1.4.3}
Two stationary solutions of \eqref{ec1.4.1} have the same distribution.
\end{proposition}

The proof of this proposition is standard, and for completeness we have included it in Appendix \ref{sec uniq ou}.

\subsection{The fractional Burgers equation}
\label{s1.6}
In this section we define what we understand by an {\em energy solution} of the fractional Burgers equation, which was introduced in \cite{GubJar} in the case of the circle as spatial state, which  is given by  
\begin{equation}
\label{ec1.5.3}
d \mc Y_t = (\mc L^\rho)^* \mc Y_t dt + m\nabla \mc Y_t^2 dt + \sqrt{2\rho(1-\rho) (-\mc L^{1/2})} d \mc W_t,
\end{equation}
where $\{\mc W_t\,;\, t \in [0,T]\}$ is an $\mathbb S'(\bb R)$-valued Brownian motion and $\mc L^\rho$ is given in \eqref{operatorLrho}.
 For that purpose,  we need to introduce various definitions.

\begin{definition}
We say that an $\mathbb S'(\bb R)$-valued stochastic process $\{\mc Y_t\,;\, t \in [0,T]\}$ is {\em Uniformly Stochastically Continuous} in $L^2(\bb R)$ (USC) if there exists a finite constant $K_0$ such that
\begin{equation}
\label{USC}
E[\mc Y_t(f)^2] \leq K_0 \|f\|^2
\end{equation}
for any $f \in \mathbb S(\bb R)$ and any $t \in [0,T]$.
\end{definition}

The USC property is satisfied by a stationary process, and the stationary case is the only one that will be considered in these notes. Note that USC is a {\em static} property, in the sense that involves only one time instant. Observe also that the USC property allows to apply Lemma \ref{l1.4.2} and therefore for any process satisfying USC, the integral
\[
\int_0^t \mc Y_s\big(\mc L^\rho f_s\big) ds
\]
is well defined for any differentiable trajectory $f: [0,T] \to \mathbb S(\bb R)$.

Now we will describe a property that involves the time evolution of the process $\{\mc Y_t\,;\, t \in [0,T]\}$. Let $\{\iota_\varepsilon\,;\, \varepsilon \in (0,1)\}$ be an approximation of the identity. An example is
\[
\iota_\varepsilon(x) = \tfrac{1}{\sqrt{2 \pi \varepsilon^2}} e^{-x^2/2\varepsilon^2},
\]
or in general $\iota_\varepsilon(x)  = \frac{1}{\varepsilon} h(\frac{x}{\varepsilon})$, where $h \in \mathbb S(\bb R)$ is positive and $\int_{\bb R} h(x) dx=1$. If the process $\{\mc Y_t; t \in [0,T]\}$ is USC, the function $h$ can even be in $L^2(\bb R) \cap L^1(\bb R)$ instead of $\mathbb S(\bb R)$.

\begin{definition}\label{def EE}
Let $\{\mc Y_t\,;\, t \in [0,T]\}$ be a given process and let us define for $\varepsilon \in (0,1)$, $s < t \in [0,T]$ and $f \in \mathbb S(\bb R)$,
\begin{equation}\label{Amacfield}
\mc A_{s,t}^\varepsilon (f) = \int_s^t \int\limits_{\bb R} \mc Y_{s'}*\iota_\varepsilon(x)^2 f'(x) dx ds'.
\end{equation}
We say that $\{\mc Y_t; t \in [0,T]\}$ satisfies an  {\em Energy Estimate} (EE) if there exist $\kappa_0>0$, $\beta \in (0,1)$ such that
\begin{equation}
\label{EE}
E\big[\big( \mc A_{s,t}^\varepsilon(f)-\mc A_{s,t}^\delta(f)\big)^2\big] \leq \kappa_0 \varepsilon (t-s)^\beta \|f'\|^2
\end{equation}
for any $f \in \mathbb S(\bb R)$, any $0<\delta < \varepsilon <1$ and any $0 \leq s < t \leq T$.
\end{definition}

Note that the energy estimate implies the existence of the limit
\begin{equation}
\label{ec1.6.1}
\mc A_{s,t}(f) = \lim_{\varepsilon \to 0} \mc A_{s,t}^\varepsilon(f)
\end{equation}
in $L^2(P)$ for any $0 \leq s<t\leq T$ and any $f \in \mathbb S(\bb R)$. This process can be understood as an integrated version of $ -\nabla \mc Y_t^2$:
\[
\mc  A_{s,t} (f) = -\int_s^t \big(\nabla \mc Y_{s'}^2\big) (f) ds.
\]

Actually, we can say more about this limit process.

\begin{proposition}
\label{p1.6.1}
Let $\{\mc Y_t; t \in [0,T]\}$ be a process satisfying \eqref{USC} and \eqref{EE} (that is, $\{\mc Y_t\,;\, t \in [0,T]\}$ is a USC process satisfying an Energy Estimate). Let $\{\mc A_{s,t}(f)\,;\, s<t \in [0,T]\}$ be the random variables obtained in \eqref{ec1.6.1}. Then, there exists an $\mathbb S'(\bb R)$-valued process $\{\mc A_t\,;\, t \in [0,T]\}$ with continuous trajectories such that\\

\begin{itemize}
\item[i)] There exists a finite constant $C$ such that for any $0\leq s< t \leq T$ and any $f \in \bb S(\bb R)$,

\begin{equation}
\label{Holder}
E\big[\big( \mc A_t(f) -\mc A_s(f) \big)^2\big] \leq C |t-s|^{1+\beta/2} \|f'\|^2
\end{equation}

and in particular
$\{\mc A_t;  t\in [0,T]\}$ is a.s.$\gamma$-H\"older continuous for any $\gamma < \frac{\beta}{4}$,\\

\item[ii)] $\mc A_{s,t}(f) = \mc A_t(f) - \mc A_s(f)$ a.s. for any $f \in \mathbb S(\bb R)$ and any $0\leq s < t \leq T$.
\end{itemize}
\end{proposition}

This proposition corresponds to Theorem 2.2 of \cite{GonJar1} for the case $\beta =1$. The proof extends easily to the case $\beta \in (0,1)$.

Finally we can define what we understand by an {\em energy solution} of the fractional Burgers equation.

\begin{definition}
We say that an $\mathbb S'(\bb R)$-valued process $\{\mc Y_t\,;\,  t \in [0,T]\}$ is an energy solution of the fractional Burgers equation
given in \eqref{ec1.5.3} if:\\
\begin{itemize}
\item[a)] the process $\{\mc Y_t\,;\, t \in [0,T]\}$ is USC and satisfies an Energy Estimate,\\

\item[b)] for any differentiable trajectory $f: [0,T] \to \bb S(\bb R)$, the process
\[
\mc M_t(f) = \mc Y_t(f) - \mc Y_0(f) - \int_0^t \mc Y_s\big( (\partial_s+\mc L^\rho) f\big) ds +m \mc A_t(f)
\]
\[
2\rho(1-\rho) \int_0^t \<f_s, -\mc L^{1/2} f_s\> ds.
\]
\end{itemize}
\end{definition}

If the process $\{\mc Y_t\,;\, t \in [0,T]\}$ is, in addition, stationary we say that $\{\mc Y_t\,;\, t \in [0,T]\}$ is a {\em stationary energy solution} of the fractional Burgers equation.
This notion of solution was proposed in \cite{GonJar1} in the context of the usual KPZ equation (that is, with $\mc L^\rho$ replaced by $\Delta$). In \cite{GubJar} it was shown the existence of energy solutions for $\mc L = -(-\Delta)^{\alpha/2}$ if $\alpha>1$ and their uniqueness if $\alpha > 10/4$.\footnote{
The theory of regularity structures of \cite{Hai1, Hai2} provides a uniqueness criterion for the stochastic Burgers equation (the case $\alpha =2$) and in principle this criterion could be extended to $\alpha$ strictly larger than $3/2$, at least in finite volume. The case $\alpha = 3/2$, which is the relevant one for these notes, seems to be out of the  reach of the current state of this  theory.
}

\subsection{The density fluctuation field}

Let $p(\cdot)$ be given by \eqref{ec1} and let $\rho \in (0,1)$. 
The density $\rho$ and the transition rate $p(\cdot)$ will be fixed from now on and up to the end of these notes. Let $\{\eta_t\,;\, t \geq 0\}$ be an exclusion process with jump rate $p(\cdot)$ and initial distribution $\mu_\rho$. Since $\mu_\rho$ is invariant, $\eta_t$ has distribution $\mu_\rho$ for any $t \geq 0$ and in particular $ E_{\mu_\rho}[\eta_t(x)]=\rho$ for any $t \geq 0$ and any $x \in \bb Z$. Let $n \in \bb N$ be a scaling parameter. We define $\eta_t^n = \eta_{t n^\alpha}$ for $t \in [0,T]$ and $n \in \bb N$. We call $\{\eta_t^n; t \in [0,T]\}$ the {\em rescaled process}. We will use the notation
\begin{equation}\label{bar_eta}
\bar{\eta}_t^n(x) = \eta_t^n(x)-\rho.
\end{equation}
We denote by $\bb P_n$ the distribution on $\mc D([0,T]\,;\, \Omega)$ of $\{\eta_t^n\,;\, t \in [0,T]\}$ starting from $\mu_\rho$ and we denote by $\bb E_n$ the expectation with respect to $\bb P_n$. The {\em density fluctuation field} is defined as the $\mathbb S'(\bb R)$-valued process $\{\mc Y_t^n\,;\, t \in [0,T]\}$ given by
\begin{equation}
\label{ec1.5.1}
\mc Y_t^n(f)  = \tfrac{1}{\sqrt n} \sum_{x\in\bb Z} \bar \eta_t^n(x) f\big(\tfrac{x-(1-2\rho) m_n^\alpha t}{n}\big)
\end{equation}
for any $t \in [0,T]$, any $n \in \bb N$ and any $f \in \mathbb S(\bb R)$. Note the Galilean transformation embedded into this definition. Recall the definition of the constant $m_n^\alpha$ given in \eqref{mnalpha}. The factor $(1-2\rho)m_n^\alpha$ is the characteristic velocity of the process $\{\eta_t^n\,;\, t \in [0,T]\}$. For this reason we say that we observe the fluctuations on {\em Lagrangian coordinates}.

The main objective of these notes is to identify the limit, as $n \to \infty$, of the fluctuation field $\{\mc Y_t^n\,; \,t \in [0,T]\}_{n\in\bb N}$. We have restricted ourselves to a finite size time window in order to avoid uninteresting topological considerations. Note that the process $\{\mc Y_t^n\,; \,t \in [0,T]\}$ has trajectories in $\mc D([0,T]\,; \,\mathbb S'(\bb R))$.
Note as well that for any $f \in \mathbb S(\bb R)$ the real-valued random variable $\{\mc Y_t^n(f)\}_{n\in\bb N}$ converges in distribution, as $n\to\infty$, to a Gaussian distribution of mean zero and variance $\rho(1-\rho) \|f\|^2$. In other words, for any $t \in [0,T]$ the sequence $\{\mc Y_t^n\}_{n \in \bb N}$ converges in distribution, as $n\to\infty$, to a white noise of variance $\rho(1-\rho)$. Note that $\mc Y_t^n$ can be understood as a random signed measure. However, the white noise can not be constructed as a random measure, which makes it more appropriate to think about $\mc Y_t^n$ as a random distribution.

\subsubsection{The case $\alpha <3/2$: \textbf{the Ornstein-Uhlenbeck equation.}}

Let $c^+$, $c^-$ be the constants associated to the transition rate $p(\cdot)$ and let $\mc L^\rho$ be the operator given by
\begin{equation}\label{operatorLrho}
\mc L^\rho f(x)  = \int\limits_{\bb R} \frac{c_\rho(y)}{|y|^{1+\alpha}} \big(f(x+y)-f(x) - \theta^\alpha(y) f'(x)\big)dy,
\end{equation}
where
\[
c_\rho(x) =
\begin{cases}
c^+ (1-\rho)+ c^- \rho; & x \geq 0\\
c^+ \rho + c^- (1-\rho); & x <0.
\end{cases}
\]
In other words, $$\mc L^\rho = \mc L(c^+ (1-\rho)+ c^- \rho,c^+ \rho + c^- (1-\rho);\alpha).$$ Note that $(\mc L^\rho)^* = \mc L^{1-\rho}$ and that the symmetric part of $\mc L^\rho$ is equal to $\mc L^{1/2}$. In particular, the symmetric part of $\mc L^{\rho}$ does not depend on $\rho$.

Now we have at our disposal all the definitions needed to state the first main result of these notes.

\begin{theorem}
\label{t1.5.1}
Let $p(\cdot)$ be as in \eqref{ec1}. Assume that $\alpha <3/2$ and that $\eta_0^n$ has distribution $\mu_\rho$. Then, the sequence of processes $\{\mc Y_t^n; t \in [0,T]\}_{n \in \bb N}$ converges in distribution, as $n\to\infty$, with respect to the $J_1$-Skorohod topology of $\mc D([0,T]; \bb S'(\bb R))$ to the stationary solution of the infinite-dimensional Ornstein-Uhlenbeck process given by
\begin{equation}
\label{ec1.5.2}
d \mc Y_t = (\mc L^\rho)^* \mc Y_t dt + \sqrt{2\rho(1-\rho)(-\mc L^{1/2})} d\mc W_t,
\end{equation}
where $\{\mc W_t; t \in [0,T]\}$ is an $\bb S'(\bb R)$-valued Brownian motion.
\end{theorem}

\subsubsection{The case $\alpha=3/2$: \textbf{ the fractional Burgers equation}}
\label{s1.5.2}
The Galilean transformation used in \eqref{ec1.5.1} has as a consequence that in the limit equation \eqref{ec1.5.2}, in spite of the transition rate $p(\cdot)$ being asymmetric, there is no transport term. It turns out that when $\alpha$ is exactly equal to $3/2$, the second-order correction of the transport term of the dynamics, non-linear in nature, has the same strength that the linear part of the dynamics. In this case, the limiting process corresponds to  the {\em fractional Burgers equation} and we can state the result as follows.

\begin{theorem}
\label{t1.6.2}
Let $p(\cdot)$ be given by \eqref{ec1}. Let us assume that $\alpha = 3/2$ and that $\eta_0^n$ has distribution $\mu_\rho$. Then the sequence of processes $\{\mc Y_t^n; t \in [0,T]\}_{n \in \bb N}$ is tight with respect to the $J_1$-Skorohod topology of $\mc D([0,T]; \bb S' (\bb R))$ and any of its limit points is a stationary energy solution
of the fractional Burgers equation \eqref{ec1.5.3}.
\end{theorem}

\begin{remark}

\quad 
\begin{enumerate}

\item  A consequence of this theorem is the existence of energy solutions of \eqref{ec1.5.3}. The method used in \cite{GubJar} restrict ourselves to finite volume.

\item In a formal way, equation \eqref{ec1.5.3} is invariant under the KPZ scaling $1:2:3$.

\item The dependence of $\mc L^\rho$ on the density $\rho$ is a new feature, not observed before in the literature.
\end{enumerate}
\end{remark}

\section{Proof of Theorems \ref{t1.5.1} and \ref{t1.6.2}}
\label{s2}

Since the limit processes on Theorems \ref{t1.5.1} and \ref{t1.6.2} are characterized by martingale problems, it is natural to begin defining various martingales related to the processes $\{\mc Y_t^n\,;\, t \in [0,T]\}_{{n\in\mathbb{N}}}$. Dynkin's formula tells us that for well behaved functions $F: [0,T] \times \Omega \to \bb R$, the process
\[
F(t,\eta_t^n)  - F(0,\eta_0^n) -\int_0^t \big(\partial_s + n^\alpha L\big)F(s,\eta_s^n) ds
\]
is a martingale, whose  quadratic variation is given by
\[
n^\alpha \int_0^t \big\{ L F(s,\eta_s^n)^2 -2 F(s,\eta_s^n) L F(s,\eta_s^n) \big\} ds.
\]
We will use this formula for functions of the form
\[
F(t,\eta_t^n)= \mc Y_t^n(f) = \tfrac{1}{\sqrt n} \sum_{x \in \bb Z} \bar{\eta}_t^n(x)f\big(\tfrac{x-(1-2\rho)m_n^\alpha t}{n}\big),
\]
where $f: [0,T] \to \mathbb S(\bb R)$ is smooth and $\bar\eta^n_t(x)$ was defined in \eqref{bar_eta}.
For this choice, let us define the martingale $\{\mc M_t^n(f); t \in [0,T]\}$ as
\begin{equation}
\label{ec2.2.1}
\mc M_t^n(f) = \mc Y_t^n(f) - \mc Y_0^n(f) - \int_0^t (\partial_s+ n^\alpha L) \mc Y_s^n(f) ds,
\end{equation}
and whose quadratic variation $\< \mc M_t^n(f)\>$ is equal to
\[
 \int_0^t n^{\alpha-1} \sum_{x,y} p(y-x) \big( \eta_s^n(y)-\eta_s^n(x) \big)^2 \big( f\big(\tfrac{y-(1-2\rho)m_n^\alpha s}{n} \big) - f\big( \tfrac{x-(1-2\rho)m_n^\alpha s}{n}\big)\big)^2 ds.
\]
Now we compute the integral part of the martingale $\mc M_t^n(f)$.  A simple computation shows that
\begin{equation*}
L \eta_s(x)=\sum_{y\in\bb Z}\Big(p(x-y)\eta_s^n(y)(1-\eta_s^n(x))-p(y-x)\eta_s^n(x)(1-\eta_s^n(y))\Big),
\end{equation*}
therefore
\begin{equation}\label{martdecomp}
n^\alpha L \mathcal{Y}_s^n(f)=\frac{n^\alpha}{\sqrt n}\sum_{x,y\in\bb Z}\Big(f\Big(\tfrac{y-(1-2\rho)m_n^\alpha s}{n}\Big)-f\Big(\tfrac{x-(1-2\rho)m_n^\alpha s}{n}\Big)\Big)p( y-x)\eta_s^n(x)(1-\eta_s^n(y)).
\end{equation}
Note that
\[
\eta_s^n(x)(1-\eta_s^n(y))-\rho(1-\rho)=(1-\rho)\bar{\eta}_s^n(x)-\rho\bar{\eta}_s^n(y)-\bar{\eta}_s^n(x)\bar{\eta}_s^n(y),
\]
where $\bar\eta^n_s(x)$ was defined in \eqref{bar_eta}.    
Using this identity and grouping together the two terms involving $x$ and $y$ in \eqref{martdecomp}, we can rewrite the right hand side of \eqref{martdecomp} as
\begin{multline*}
\frac{n^\alpha}{\sqrt n}\sum_{x,y\in\bb Z}\Big(f\Big(\tfrac{y-(1-2\rho)m_n^\alpha s}{n}\Big)-f\Big(\tfrac{x-(1-2\rho)m_n^\alpha s}{n}\Big)\Big)\Big\{(1-\rho)p(y-x)+\rho p(x-y)\Big\}\bar{\eta}_s^n(x)\\
-\frac{n^\alpha}{\sqrt n}\sum_{x,y\in\bb Z}\Big(f\Big(\tfrac{y-(1-2\rho)m_n^\alpha s}{n}\Big)-f_t\Big(\tfrac{x-(1-2\rho)m_n^\alpha s}{n}\Big)\Big)p(y-x)\bar{\eta}_s^n(x)\bar{\eta}_s^n(y).
\end{multline*}

On the other hand, 
\begin{equation*}
\partial_s \mathcal{Y}_s^n(f)=-(1-2\rho)\frac{m_n^\alpha}{n\sqrt n}\sum_{x\in\bb Z}\partial_s f\Big(\tfrac{x-(1-2\rho)m_n^\alpha s}{n}\Big)\bar\eta^n_s(x).
\end{equation*}
Moreover, a simple computation shows that the right hand side of the previous expression equals to
\begin{equation}
-\frac{n^\alpha}{\sqrt n}\sum_{x,y\in\bb Z}((1-\rho)p(y-x)+\rho p(x-y))\theta^\alpha\Big(\frac{y-x}{n}\Big)\partial_s f\Big(\tfrac{x-(1-2\rho)m_n^\alpha s}{n}\Big)\bar\eta^n_s(x),
\end{equation}
where $\theta^\alpha(\cdot)$ was defined in \eqref{eq2}.
Putting together the previous computations, 
the integral part of the martingale can be written as
\begin{equation}\label{int_part_mart}
\begin{split}
&\int_0^t \frac{n^\alpha}{\sqrt n}\sum_{x,y\in\bb Z}\Big(f\Big(\tfrac{y-(1-2\rho)m_n^\alpha s}{n}\Big)
-f\Big(\tfrac{x-(1-2\rho)m_n^\alpha s}{n}\Big)
-\theta^\alpha\Big(\frac{y-x}{n}\Big)\partial_sf\Big(\tfrac{x-(1-2\rho)m_n^\alpha s}{n}\Big)\Big)
\\
& \quad \quad \quad \quad \quad \quad \quad \quad \quad \quad\times 
\Big\{(1-\rho)p(y-x)+\rho p(x-y)\Big\}\bar{\eta}_s^n(x)ds\\
-&\int_0^t \frac{n^\alpha}{\sqrt n}\sum_{x,y\in\bb Z}\Big(f\Big(\tfrac{y-(1-2\rho)m_n^\alpha s}{n}\Big)
-f\Big(\tfrac{x-(1-2\rho)m_n^\alpha s}{n}\Big)\Big)
p(y-x)\bar{\eta}_s^n(x)\bar{\eta}_s^n(y) ds.
\end{split}
\end{equation}
Now, for a smooth function $f: \bb R \to \bb R$, let  $\mc L_n^\rho $ be the operator defined by
\[
\begin{split}
\mc L_n^\rho f \big(\tfrac{x}{n}\big)&= n^\alpha \sum_{y\in\bb Z} \Big\{ ((1-\rho) p(y-x) + \rho p(x-y))  \Big( f\Big(\tfrac{y}{n}\Big)- f\big( \tfrac{x}{n}\big)-
	 \theta^\alpha \big( \tfrac{y-x}{n}\big) f' \big( \tfrac{x}{n} \big)\Big)\Big\}.\\
\end{split}
\]
Then, by  a change of variables
we get that the  first term at the right hand side of \eqref{int_part_mart} coincides with $$\int_0^t \mc Y_s^n(\mc L_n^\rho f)ds.$$
We also note that by a simple computation
we can rewrite  $\mc L_n^\rho f$ as
\begin{equation}\label{operator L n rho}
\begin{split}
\mc L_n^\rho f \big(\tfrac{x}{n}\big)
	&= n^\alpha \sum_{y\in\bb Z}  \Big\{ ((1-\rho) p(y-x) + \rho p(x-y))  \big( f\big(\tfrac{y}{n} \big) - f\big( \tfrac{x}{n}\big)\big) \Big\} -(1-2\rho) \frac{m_n^\alpha}{n} f'\big( \tfrac{x}{n}).
	\end{split}
	\end{equation}

Now we look at the second term at the right hand side of \eqref{int_part_mart}.  By the anti-symmetry of $p(\cdot)$, we can replace there $p(\cdot)$ by $a(\cdot)$.  For that purpose note that by exchanging $x$ with $y$ we have that
\begin{equation*}
\begin{split}
 &\frac{1}{2}\sum_{x,y} \Big(f\Big(\tfrac{y-(1-2\rho)m_n^\alpha s}{n}\Big)
-f\Big(\tfrac{x-(1-2\rho)m_n^\alpha s}{n}\Big)\Big)p(y-x) \bar{\eta}_s^n(y) \bar{\eta}_s^n(x) \\
 +&\frac{1}{2}\sum_{x,y}p(x-y) \bar{\eta}_s^n(y) \bar{\eta}_s^n(x) \Big(f\Big(\tfrac{x-(1-2\rho)m_n^\alpha s}{n}\Big)
-f\Big(\tfrac{y-(1-2\rho)m_n^\alpha s}{n}\Big)\Big)\\
 =&\sum_{x,y} a(x-y) \bar{\eta}_s^n(y) \bar{\eta}_s^n(x) \big( f\big(\tfrac{x-(1-2\rho)m_n^\alpha s}{n}\big) - f \big( \tfrac{y-(1-2\rho)m_n^\alpha s}{n}\big)\big).
 \end{split}
\end{equation*}
From the previous  computations, the martingale $\mc M_t^n(f)$ can be written in a more explicit way as
\begin{equation}
\label{eq3.1}
\mc M_t^n(f) = \mc Y_t^n(f) - \mc Y_0^n(f) -\int_0^t \mc Y_s^n(\mc L_n^\rho f)ds + A_t^n(f),
\end{equation} where $\mc L_n^\rho$ was defined in  \eqref{operator L n rho} and 
 the process $\{A_t^n(f)\,;\, t \geq 0\}$ is defined  as
 \begin{equation}\label{Afield}
A_t^n(f) = \int_0^t n^{\alpha -1/2} \sum_{x,y} a(y-x) \bar{\eta}_s^n(y) \bar{\eta}_s^n(x) \big( f\big(\tfrac{y-(1-2\rho)m_n^\alpha s}{n}\big) - f \big( \tfrac{x-(1-2\rho)m_n^\alpha s}{n}\big)\big)ds.
\end{equation}
Note that Proposition \ref{p2} also holds for the operators $\mc L_n^\rho$ and $\mc L^\rho$, therefore $\mc L_n^\rho$ converges to $\mc L^\rho$, as $n \to \infty$, in the sense described there. We will use indistinctly the symbol $f$ for a function $f \in \mathbb S(\bb R)$ and for the function from $[0,T]$ to $\mathbb S(\bb R)$ with constant value equal to $f$.\\
The proof of Theorems \ref{t1.5.1} and \ref{t1.6.2} follows the classical structure of convergence in distribution of stochastic processes.  The proofs of the theorems  are not very different among them, the main difference is that we will show that $A_t^n(f)$ converges to $0$, as $n\to\infty$, if $\alpha <3/2$ and that it is asymptotically equivalent to a functional of the fluctuation field $\mc Y_t^n$, if $\alpha =3/2$. \\
 We start by proving  tightness with respect to the corresponding topologies, then we prove that any limit point is a solution of the corresponding martingale problem. In the case of Theorem \ref{t1.5.1}, the uniqueness criterion of Proposition \ref{p1.4.3} allows to conclude the desired result. In the case of Theorem \ref{t1.6.2}, the lack of an uniqueness criterion as the one stated in Proposition \ref{p1.4.3} restrict ourselves to convergence through subsequences.

\section{Tightness}
\label{s3}
In order to prove tightness of  the sequence $\{\mc Y_t^n\,; \,t \in [0,T]\}_{n \in \bb N}$,  first we recall  the so-called {\em Mitoma's criterion}, which reduces the proof of tightness of distribution-valued processes to the proof of tightness for real-valued processes.

\begin{proposition}[Mitoma's criterion \cite{Mit}]
\label{p2.2.1}
A sequence $\{\mc Y_t^n; t \in [0,T]\}_{n \in \bb N}$ of stochastic processes with trajectories in $\mc D([0,T]; \mathbb S'(\bb R))$ is tight with respect to the $J_1$-Skorohod topology if, and only if, the sequence of real-valued processes $\{\mathbb Y_t^n(f); t \in [0,T]\}_{n \in \bb N}$ is tight with respect to the $J_1$-Skorohod topology of $\mc D([0,T]; \bb R)$ for any $f \in \mathbb S(\bb R)$.
\end{proposition}
From the previous proposition, in order to prove tightness of the sequence $\{\mc Y_t^n\,; \,t \in [0,T]\}_{n \in \bb N}$
it is enough to show tightness of the sequence of real-valued processes $\{\mc Y_t^n(f) \,;  \,t \in [0,T]\}_{n \in  \bb N}$ for any $f \in \bb S(\bb R)$. According to \eqref{eq3.1}, it is enough to show that the processes
\begin{equation*}
\begin{split}
&\{\mc Y_0^n(f)\}_{n \in \bb N},  \quad \Big\{\int_0^t \mc Y_s^n(\mc L_n^\rho f) ds; t \in [0,T]\Big\}_{n \in \bb N}, \\& \{A_t^n(f); t \in [0,T]\}_{n \in \bb N},  \quad \textrm{and} \quad  \{\mc M_t^n(f); t \in [0,T]\}_{n \in \bb N}
\end{split}
\end{equation*}
are tight.
\subsection{Tightness of the initial field}
\label{sec:tight_ini_field}
We prove, in fact, that  the sequence $\{\mc Y_0^n(f)\}_{n\in \bb N}$ converges as $n\to\infty$. Computing the characteristic function of $\mc Y_0^n(f)$, we can check that it converges in distribution, as $n\to\infty$,  to a Gaussian law of mean $0$ and variance $\rho(1-\rho) \|f\|^2$. In particular, $\{\mc Y_0^n(f)\}_{n \in \bb N}$ is tight.

\subsection{Tightness of the martingales}
\label{sec:tight_mart}
In order to prove tightness of the martingales $\{\mc M_t^n(f); t \in [0,T]\}_{n \in \bb N}$ we will in fact prove that they converge. For that, we will use the following {\em convergence criterion} (see Theorem 2.1 of \cite{Whi}):

\begin{proposition}
\label{p2.2.2}
A sequence $\{M_t^n\,;\, t \in [0,T]\}_{n\in \bb N}$ of square-integrable martingales converges in distribution, with respect to the $J_1$-Skorohod topology of $\mc D([0,T]; \bb R)$, as $n\to\infty$, to a Brownian motion of variance $\sigma^2$ if:
\begin{itemize}
\item[i)] {\em Asymptotically negligible jumps:} for any $\varepsilon >0$,
\[
\lim_{n \to \infty} P \Big( \sup_{0 \leq t \leq T} \big| M_{t-}^n-M_t^n\big| > \varepsilon\big) =0,
\]
\item[ii)] {\em Convergence of quadratic variations:} for any $t \in [0,T]$,
\[
\lim_{n \to \infty} \< M_t^n\> = \sigma^2 t,
\]
in distribution.
\end{itemize}
\end{proposition}
 In our present situation, to check i) of the previous proposition  we first note that the jumps of $\mc M_t^n(f)$ are the same of $\mc Y_t^n(f)$, since the other terms in \eqref{eq3.1} are continuous. Therefore,
\[
\sup_{t \in [0,T]} \big|\mc M_t^n(f) - \mc M_{t-}^n(f)\big| \leq \frac{\|f\|_\infty}{\sqrt n}
\]
and the jumps of $\{\mc M_t^n(f); t \in [0,T]\}$ are asymptotically negligible. Now to prove ii), we first note  that
\begin{equation*}
\begin{split}
\bb E_n \Big[ n^{\alpha-1} &\sum_{x,y} p(y-x) \big(\eta_s^n(y) -\eta_t^n(x)\big)^2 \big(f\big(\tfrac{y}{n}\big) - f\big(\tfrac{x}{n}\big)\big)^2\Big] \\
&= \frac{n^{\alpha-1}}{2} \sum_{x-y>0} (p(y-x)+p(x-y)) 2\rho(1-\rho)\big(f\big(\tfrac{y}{n}\big) - f\big(\tfrac{x}{n}\big)\big)^2\\
&\quad    \quad +\frac{n^{\alpha-1}}{2} \sum_{x-y<0} (p(y-x)+p(x-y)) 2\rho(1-\rho)\big(f\big(\tfrac{y}{n}\big) - f\big(\tfrac{x}{n}\big)\big)^2\\
&=n^{\alpha-1} \sum_{x,y} s(y-x) 2\rho(1-\rho) \big(f\big(\tfrac{y}{n}\big) - f\big(\tfrac{x}{n}\big)\big)^2\\
		&= 4\rho(1-\rho) \mc E_n(f) \xrightarrow{n \to \infty} 4\rho(1-\rho) \mc E(f),
\end{split}
\end{equation*}
and this shows that $\bb E_n [\<\mathcal {M}_t^n(f)\>]$ converges, as $n\to\infty$, to $4\rho(1-\rho) \mc E(f)$. Now we prove that  $\bb E[(\<\mathcal {M}_t^n(f)\>-\bb E_n [\<\mathcal {M}_t^n(f)\>])^2]$ vanishes, as $n\to\infty$, which will imply ii) of the previous proposition. 
A simple computation shows that the variance of $\<\mathcal {M}_t^n(f)\>$ is equal to\footnote{
Along these notes, we will denote by $c_i(\rho)$ various constants which depend only on $\rho$. The exact value of these constants will not be important; only the fact that they depend only on $\rho$.
}
\begin{equation}\label{imp_terms}
\begin{split}
&c_1(\rho) n^{2\alpha -2} \sum_{x,y} s(y-x)^2 \big(f\big(\tfrac{y}{n}\big) - f\big(\tfrac{x}{n}\big) \big)^4\\
	+&c_2(\rho)  n^{2\alpha-2} \sum_{x} \Big(\sum_{y} s(y-x)\big(f\big(\tfrac{y}{n}\big)-f\big(\tfrac{x}{n}\big)\big)^2\Big)^2.
	\end{split}
\end{equation}
From the computations of Appendix \ref{sec_a1}, the first term of last expression is of order $n^{2\alpha-5}$, while the second one 
 is of order  $n^{\alpha-3}$. Therefore, both vanish as $n\to
 \infty$. As a consequence we get that
\[
\lim_{n \to \infty} \bb E_n \Big[ \Big(\<\mc M_t^n(f)\big>-4\rho(1-\rho) \mc E(f) t\Big)^2\Big]=0,
\]
which shows  that the martingales $\{\mc M_t^n(f); t \in [0,T]\}_{n \in \bb N}$ converge in distribution, as $n\to\infty$, to a Brownian motion of variance $4\rho(1-\rho) \mc E(f)$.

\subsection{Tightness of the integral terms}\label{sec:tight_integral}
Now we have to prove tightness of the integral terms $\{\int_0^t \mc Y_s^n(\mc L_n^\rho f) ds; t \in [0,T]\}_{n \in \bb N}$. For that purpose we will use the Kolmogorov-Centsov's tightness criterion that we state as follows.
\begin{proposition}[Kolmogorov-Centsov's tightness criterion]

\label{KCC}
A sequence of continuous processes $\{X_t^n\,;\, t \in [0,T]\}_{n \in \bb N}$ is tight, with respect to the uniform topology of $\mc C([0,T]; \bb R)$, if there exist constants $K, a,b >0$ such that
\begin{equation}
\label{ecKCC}
E\big[ \big|X_t^n -X_s^n\big|^a \big] \leq K |t-s|^{1+b}
\end{equation}
for any $s,t \in [0,T]$ and any $n \in \bb N$. If the processes $\{X_t^n\,;\, t \in [0,T]\}_{n \in \bb N}$ are stationary, it is enough to verify that
\[
E\big[\big|X_t^n-X_0^n\big|^a\big] \leq Kt^{1+b}.
\]
\end{proposition}
In our present situation we want to apply last result to integral processes. Combining it with the  Cauchy-Schwarz inequality we have that 
\[
E\Big[ \Big( \int_s^t x_n(s') ds' \Big)^2\Big] \leq K |t-s|^2
\] and from this we have the criterion that we will employ in what follows. 
\begin{proposition}
\label{p2.2.3}
A sequence of processes of the form $\Big\{ \int_0^t x_n(s) ds\,;\, t \in [0,T]\Big\}_{n\in\mathbb{N}}$ is tight, with respect to the uniform topology in $\mc C([0,T]; \bb R)$, if
\[
\lim_{n \to \infty} \sup_{0 \leq t \leq T} E [ x_n(t)^2] < +\infty.
\]
\end{proposition}

Now, in order to check that   $\{\int_0^t \mc Y_s^n(\mc L_n^\rho f) ds; t \in [0,T]\}_{n \in \bb N}$ is tight, first, we observe  that
\[
\bb E_n \big[ \mc Y_s^n(\mc L_n^\rho f)^2\big] = \frac{\rho(1-\rho)}{n} \sum_{x} \big(\mc L_n^\rho f\big(\tfrac{x}{n}\big)\big)^2.
\]
By Proposition \ref{p2} this sum converges, as $n\to\infty$, to $\rho(1-\rho) \int_{\mathbb{R}} (\mc L^\rho f)^2(x) dx$, and therefore the hypotheses of Proposition \ref{p2.2.3} are satisfied. This shows  that  the sequence $\{\int_0^t \mc Y_s^n(\mc L_n^\rho f) ds\,;\, t \in [0,T]\}_{n \in \bb N}$ is tight. We remark that Proposition \ref{p2} is stated for different operators but the same result holds in the cases considered here. 

Note that all the previous results hold for any $\alpha\leq 3/2$.

\subsection{Tightness of $\{A_t^n(f); t \in [0,T]\}_{n \in \bb N}$}\label{sec:tight_a_field}
 This is  really the difficult term. The problem comes from the fact that the spatial normalization is $n^{\alpha-3/2}$ instead of $n^{-1/2}$. Therefore, for $\alpha >1$ we need to make efficient use of the time integration in order to show  tightness of this term. For $\alpha <3/2$ we will see that this term is asymptotically negligible, while for $\alpha = 3/2$ we will show that  it is asymptotically equivalent to a function of the density fluctuation field. Since the arguments to prove tightness of $\{A_t^n(f); t \in [0,T]\}_{n \in \bb N}$ depend on the regime of $\alpha$ we  devote separate sections for them.

\subsubsection {{Tightness of $\{A_t^n(f); t \in [0,T]\}_{n \in \bb N}$:\bf{ the case $\alpha \leq 1$.}}}\label{s3.1} 

\quad

\quad

Recall the definition of $A_t^n(f)$ from \eqref{Afield}.
Note that
\begin{equation}\label{int:field_a}
\begin{split}
\bb E_n \Big[ \Big( n^{\alpha -1/2}& \sum_{x,y} a(y-x) \bar \eta_t^n(x) \bar \eta_t^n(x) \big(f\big(\tfrac{y}{n}\big) -f \big( \tfrac{x}{n}\big)\big)\Big)^2 \Big]\\
		=& c_3(\rho) n^{2\alpha -1} \sum_{x,y} a(y-x)^2 \big(f\big(\tfrac{y}{n}\big) -f \big( \tfrac{x}{n}\big)\big)^2.\end{split}
\end{equation}
From the computations of Appendix \ref{sec_a2}, if $\alpha <1$ then the expectation above is of order $o(1)$; while if $\alpha =1$  it is of order $O(1)$. In any case, by the compactness criterion of Proposition \ref{p2.2.3} we conclude that the sequence $\{A_t^n(f); t \in [0,T]\}_{n \in \bb N}$ is tight for any $\alpha\leq 1$.

With last result we  end the proof of tightness of $\{\mc Y_t^n; t \in [0,T]\}_{n \in \bb N}$ in the case $\alpha \leq 1$. 

\subsubsection{{Tightness of $\{A_t^n(f); t \in [0,T]\}_{n \in \bb N}$: \bf{the case $1<\alpha <1+\frac{2}{5+\sqrt{33}}$}}}\label{s3.2}

\quad

\quad

Below, unless explicitly stated, we {\em do not} assume $1< \alpha <1+\frac{2}{5+\sqrt{33}}$ on the computations made in this section. The proof of tightness given in this section is quite technical but we note that it will rely on several applications of Proposition \ref{p2.2.3}. Therefore we will be interested in obtaining upper bounds for the second moments of $A_t^n(f)$. To do that successfully,  we split $A_t^n(f)$ into several intermediate additive functionals of the process, and at each step we will obtain upper bounds on the second moments of these functionals, from where we will conclude tightness of $A_t^n(f)$.

The first step in this procedure consists in noting that the sum in the definition of $A_t^n(f)$  can be restricted to $|y-x| \leq K_n$ for $K_n \gg n^{\frac{2\alpha-2}{2\alpha-1}}$ thanks to the next lemma whose proof is given in Appendix \ref{sec_a3}.

\begin{lemma}
\label{lem_1} For $K_n \gg n^{\frac{2\alpha-2}{2\alpha-1}}$ we have that
\begin{equation}\label{int_3}
\lim_{n \to \infty} \bb E_n \Big[ \Big( n^{\alpha - 1/2} \!\!\!\!\sum_{|y-x| \geq K_n}\!\!\!\! a(y-x) \big( f\big( \tfrac{y}{n} \big) -f \big( \tfrac{x}{n} \big) \big) \bar{\eta}_t^n(x) \bar{\eta}_t^n(y)\Big)^2\Big] =0.
\end{equation}
\end{lemma}
As a consequence of the previous lemma, we can take any sequence $\{K_n; n \in \bb N\}$ satisfying the condition $K_n \gg n^{\frac{2\alpha-2}{2\alpha-1}}$  and we can restrict the sum in the definition of $A_t^n(f)$ to $|y-x| \leq K_n$; the rest of the sum is tight by Proposition \ref{p2.2.3}. Note that $K_n = o(n)$. In order to simplify the notation, we will drop the subscript $n$ from $K_n$.  Now we use  \eqref{useful_bound}, and we have, for $x,y \in \bb Z$ such that $|y-x| \leq K$, that
\[
f\big(\tfrac{y}{n} \big) - f\big( \tfrac{x}{n} \big) \leq  \tfrac{y-x}{n} f' \big( \tfrac{x}{n} \big) + \tfrac{(y-x)^2}{n^2} \|f''((y-x)/n)\|_{K/n,\infty},
\]
where $\|f'' (u)\|_{M,\infty}=\sup_{|z-u|\leq M}|f''(z)|$. At this point, we put the previous equality back into $A_t^n(f)$ and we use the next lemma, whose proof is given in  Appendix \ref{sec_a4}.
\begin{lemma}
\label{lem_2} For $K \gg n^{\frac{2\alpha-4}{2\alpha-3}}$ we have that
\begin{equation}\label{int_4}
\lim_{n \to \infty} \bb E_n \Big[ \Big( n^{\alpha-1/2} \!\!\!\!\sum_{|y-x| \leq K}\!\!\!\! a(y-x) \bar{\eta}_t^n(x) \bar{\eta}_t(y)^n \frac{(y-x)^2}{n^2}  \|f''((y-x)/n)\|_{K/n,\infty}\Big)^2\Big]=0.
\end{equation}
\end{lemma}
Since from Lemma \ref{lem_1} we have $K \gg n^{\frac{2\alpha-2}{2\alpha-1}}$, for this choice of $K$, \eqref{int_4} follows.  
Now,  we have to obtain upper an upper bound for
\begin{equation}
\label{ec3.2.2_new_new}
\bb E_n\Big[\Big(\int_0^t n^{\alpha -3/2} \!\!\!\! \sum_{|y-x| \leq K}\!\!\!\! (y-x) a(y-x)  \bar{\eta}_s^n(x) \bar{\eta}_s^n(y) f'\big(\tfrac{x}{n}\big) ds\Big)^2\Big].
\end{equation}

For that purpose, we will have to use more sophisticated estimates than the ones we used before. To state this estimates properly, we have to introduce  some notation.

For each $\sigma\in [0,1]$, let $L^2(\mu_\sigma)$ be the Hilbert space associated to the measure $\mu_\sigma$, that is, the space of functions $F: \Omega \to \bb R$ such that $\int  F^2 d \mu_\sigma < +\infty$. We denote by $\<F,G\>_\sigma$ the inner product in $L^2(\mu_\sigma)$. For $F \in L^2(\mu_\sigma)$ we define
\[
\|F\|^2_{-1} = \sup_{G \text{ local}} \big\{ 2 \<F,G\>_\sigma - \<G,-n^\alpha L G\>_\sigma\big\}.
\]
Note that $\|F\|^2_{-1}=+\infty$ if $\int F d\mu_\sigma \neq 0$. The relevance of this quantity is given by the following inequality:

\begin{proposition}[Kipnis-Varadhan inequality]
\label{p2.3.1}
Assume that $\eta_0^n$ has distribution $\mu_\sigma$ and let $F: [0,T] \to L^2(\mu_\sigma)$.  Then,
\[
\bb E_n\Big[ \Big( \sup_{0 \leq t \leq T} \int_0^t F(s,\eta_s^n) ds \Big)^2\Big] \leq 14 \int_0^T \|F(t,\cdot)\|_{-1}^2 dt.
\]
\end{proposition}

This kind of inequality was introduced in \cite{KipVar} in the context of stationary, reversible Markov chains. A proof of this inequality in the version stated above can be found in \cite{ChaLanOll}. In order to make effective use of Proposition \ref{p2.3.1} we need to know how to estimate $\|F\|_{-1}^2$ at least for a class of functions large enough. This is the context of the following proposition:

\begin{proposition}
\label{p2.3.2}
Let $m \in \bb N$ and let $k_0<\dots< k_m$ be a sequence on $\bb Z$. Let $\{F_1,\dots,F_m\}$ be a sequence of local functions such that $\supp(F_i) \subseteq \{k_{i-1}+1,\dots,k_i\}$ for any $i \in \{1,\dots,m\}$. Let us define $\ell_i = k_i - k_{i-1}$ for $i=1,\dots,m$.\footnote{
Note that the support of $F_i$ has at most diameter $\ell_i$.
}
Assume that $\int F_i d \mu_\sigma=0$ for any $\sigma \in [0,1]$ and any $1 \leq i \leq m$. Then, for any $\sigma \in [0,1]$
\[
\big\| F_1 + \dots + F_m \big\|^2_{-1} \leq \kappa \sum_{i=1}^m\frac{\ell_i^\alpha}{n^\alpha} \int F_i^2 d\mu_\sigma.
\]
\end{proposition}

When $p(\cdot)$ is the jump rate of a simple random walk, that is $\alpha =2$, this proposition is exactly Proposition 7 in \cite{GonJar1}. In our case, the proof is practically identical to the proof of that proposition, therefore we omitted it. Combining Propositions \ref{p2.3.1} and \ref{p2.3.2} we obtain the following inequality:
\begin{proposition}
\label{p2.3.4} Let $\{F_1,\dots,F_m\}$ be as in Proposition \ref{p2.3.2}. Then, for $F= F_1+\dots F_m$,
\[
\bb E_n \Big[ \Big( \sup_{0 \leq t \leq T} \int_0^t F(s,\eta^n_s) ds \Big)^2 \Big] \leq 14 \kappa \int_0^T \sum_{i=1}^m \frac{\ell_i^\alpha}{n^\alpha} \int F_i^2(s,\eta) d\mu_\sigma\,ds .
\]
\end{proposition}

In what follows, we will use repeatedly Proposition \ref{p2.3.4}; and we note that Propositions \ref{p2.3.1} and \ref{p2.3.2} are needed only to prove this proposition.

Recall that up to now, the idea was to reduce $A_t^n(f)$ to a sum of variables with the smallest possible support. The point is that the smaller the support of the functions involved in, the more effective Proposition \ref{p2.3.4} is. Proposition \ref{p2.3.4} does not apply directly to \eqref{ec3.2.2_new_new} for two reasons. First, the supports of the functions $\bar{\eta}^n(x) \bar{\eta}^n(y)$ are intertwined. This problem can be solved dividing the sum in \eqref{ec3.2.2_new_new} into $K$ sums of functions with disjoint supports. And second, the functions $\bar{\eta}^n(x) \bar{\eta}^n(y)$ do not have mean zero for all invariant measures $\mu_\sigma$. A strategy to solve the second problem is to add and subtract the function $\psi_x^{K}(\eta)$, where for any $2\leq \ell  \in \bb N$, any $x \in \bb Z$, any $\eta \in \Omega$ and for  $\rho\in(0,1)$ we have  $\psi_x^\ell: \Omega \to \bb R$ given by
\[
\psi_x^\ell(\eta) = E \big[ \bar{\eta}(x) \bar{\eta}(x+1) \big| \eta^\ell(x)\big],
\]
where the conditional expectation is taken with respect to the measure $\mu_\rho$ and 
\[
\eta^\ell(x) = \frac{1}{\ell} \sum_{i=0}^{\ell-1} \eta(x+i).
\]
Do not confuse the {\em function} $\eta^\ell(x): \Omega \to \bb R$ with the {\em process} $\eta_t^n(x)$. We will not use both notations together; the risk of confusion will be minimal. An explicit computation shows that
\begin{equation}
\label{ee}
\psi_x^\ell(\eta) = \frac{\ell}{\ell-1} \Big\{ \big(\bar{\eta}^\ell(x)\big)^2 - \frac{\rho(1-\rho)}{\ell} \Big\} + \frac{2\rho-1}{\ell-1} \bar{\eta}^\ell(x),
\end{equation}
and in particular
\begin{equation}
\label{eedes}
\begin{gathered}
\int \psi_x^\ell(\eta)^2 d\mu_\rho \leq \frac{c(\rho)}{\ell^2},\\
\int \Big( \psi_x^\ell(\eta) - \big(\eta^\ell(x) -\rho\big)^2 +\frac{\rho(1-\rho)}{\ell}\Big)^2 d\mu_\rho \leq \frac{c(\rho)}{\ell^3}.
\end{gathered}
\end{equation}
Now, define $\mc Z_j = \{K z+j; z \in \bb Z\}$. Then, by the inequality $(x+y)^2\leq 2x^2+2y^2$,   \eqref{ec3.2.2_new_new} is bounded by
\begin{multline}
\label{ec3.2.3}
2 \bb E_n \Big[ \Big( \int_0^t n^{\alpha -3/2} \!\!\sum_{\substack{x \in \mc Z_j \\ j=1,\dots,K}} \!\! \sum_{y=1}^{K-1} y a(y) \big\{ \bar{\eta}_s^n(x) \bar{\eta}_s^n(x+y) - \psi_x^{K} (\eta_s^n)\big\} f'\big(\tfrac{x}{n} \big) ds \Big)^2 \Big] \\
		+ 2\bb E_n \Big[ \Big( \int_0^t n^{\alpha-3/2} \sum_{\substack{x \in \mc Z_j\\j=1,\dots,K}} \sum_{y=1}^{K-1} ya(y) \psi_x^K(\eta_s^n) f'\big(\tfrac{x}{n} \big) ds\Big)^2 \Big].
\end{multline}
Now we use the two next lemmas whose proofs are given in Appendix \ref{sec_a5}.
\begin{lemma}
\label{lem_3} We have that
\begin{equation}
\label{ec3.2.2}
\bb E_n \Big[ \Big( \int_0^t n^{\alpha-3/2} \sum_{\substack{x \in \mc Z_j\\j=1,\dots,K}} \sum_{y=1}^{K-1} ya(y) \psi_x^K(\eta_s^n) f'\big(\tfrac{x}{n} \big) ds\Big)^2 \Big]\leq t^2\frac{n^{2\alpha-2}}{K},
\end{equation}
so that it vanishes, as $n\to\infty$, if $K\gg n^{2\alpha-2}.$
\end{lemma}
\begin{lemma}
\label{lem_4} We have that
\begin{equation}
\label{ec3.2.2_new}
\lim_{n\to\infty}\bb E_n \Big[ \Big( \int_0^t n^{\alpha -3/2} \!\!\sum_{\substack{x \in \mc Z_j \\ j=1,\dots,K}} \!\! \sum_{y=1}^{K-1} y a(y) \big\{ \bar{\eta}_s^n(x) \bar{\eta}_s^n(x+y) - \psi_x^{K} (\eta_s^n)\big\} f'\big(\tfrac{x}{n} \big) ds \Big)^2 \Big]\leq t\frac{K^{1+\alpha}}{n^{2-\alpha}},
\end{equation}
which vanishes, as $n\to \infty$, if  
$K \ll n^{\frac{2-\alpha}{1+\alpha}}$.
\end{lemma}

At this point we collect the estimates we have so far on $K$:  $$K \gg n^{\frac{2\alpha-2}{2\alpha-1}},\quad  K \gg n^{2\alpha-2} \quad \textrm{and}\quad K \ll n^{\frac{2-\alpha}{1+\alpha}}.$$ Since $\alpha>1$ we are therefore reduced to $n^{2\alpha-2}\ll K \ll n^{\frac{2-\alpha}{1+\alpha}}$.
Now, if $\alpha < 1+ \frac{2}{5+\sqrt{33}}$, there exists $\gamma$ such that
\[
2\alpha -2 < \gamma < \frac{2-\alpha}{1+\alpha}.
\]
Note that the value $1+ \frac{2}{5+\sqrt{33}}$, comes from the fact that we need to have $\alpha>1$ such that $2\alpha -2 <  \frac{2-\alpha}{1+\alpha}.$
Finally, looking at the bounds obtained in  \eqref{ec3.2.2} and \eqref{ec3.2.2_new}, by taking  $K= n^{\frac{\alpha}{\alpha+2}}$ and replacing this in both \eqref{ec3.2.2} and \eqref{ec3.2.2_new}, we conclude that the variance of \eqref{ec3.2.2} is bounded by
\[
C(f,\rho) tn^{\theta},
\]
where $\theta = \frac{2\alpha^2+\alpha-4}{2+\alpha}$. Note that $\theta < 0 $ if $\alpha < 1+ \frac{2}{5+\sqrt{ 33}}$.  Therefore \eqref{ec3.2.2} vanishes in $\bb{L}^2(\bb P_n)$, as $n\to\infty$. However, this is not enough to prove tightness, since in order to apply Proposition \ref{KCC} we need an exponent bigger than one on $t$. For that purpose, we apply the Cauchy-Schwarz inequality and we perform similar computations to those of Appendix \ref{sec_a2} to obtain a rough bound for the variance of \eqref{ec3.2.2} given by 
\[
C(f,\rho) t^2 n^{2\alpha-2}.
\]
Now, the following lemma will be useful.
\begin{lemma}
\label{lvar}
For any $a, b>0$, there exist $C, \delta, \varepsilon>0$ such that
\[
\min\Big\{\frac{t}{n^a}, t^2 n^b\Big\} \leq \frac{C t^{1+\delta}}{n^\varepsilon}.
\]
\end{lemma}
The proof is elementary and for that reason we omit it. Using this lemma we conclude that, for $1 < \alpha < 1+ \frac{2}{5+\sqrt{33}}$, the variance of \eqref{ec3.2.2} is bounded by $Ct^{1+\delta}{n^{-\varepsilon}}$. 

By Proposition  \ref{p2.2.3} we conclude two things. First that $\{\mc A_t^n(f); t \in [0,T]\}_{n \in \bb N}$ is tight and second that any of its limit points are identically null. In order to cross the barrier $\alpha < 1+\frac{2}{5+\sqrt{33}}$, we will need to perform a multiscale analysis, which is the content of the next section.

\subsubsection{{Tightness of $\{A_t^n(f); t \in [0,T]\}_{n \in \bb N}$:  \bf{the case $1+\frac{2}{5+\sqrt{33}} \leq \alpha <\frac{3}{2}$}}}
\label{s3.3}

\quad

\quad

From Lemma \ref{lem_1} we could reduce the sum defining $A_t^n(f)$ to a sum over $|y-x| \leq K$, where $K \gg n^{\frac{2\alpha-2}{2\alpha-1}}$. Then, in Lemma \ref{lem_4} we observed that the first term in \eqref{ec3.2.3} converges to $0$ if $K \ll n^{\frac{2-\alpha}{1+\alpha}}$. Now, we need to see what can we say about the first term in \eqref{ec3.2.3} when we have  the sum  $L\leq |y-x|\leq K$, where: $$L \gg n^{\frac{2\alpha-2}{2\alpha-1}}\quad \textrm{and \quad} K \ll n^{\frac{2-\alpha}{1+\alpha}}.$$ In order to do that, let $L<K$ be given, and note that by repeating the arguments of the proof of Lemma \ref{lem_4} we have that 
\begin{equation}
\label{est1}
\begin{split}
 \bb E_n\Big[\Big(\int_0^t n^{\alpha-3/2} &\sum_{\substack{x \in \mc Z_j \\ j=1,\dots,K}} \sum_{y=L}^{K-1} y a(y) \Big(\bar{\eta}_s^n(x) \bar{\eta}_s^n(x+y) - \psi_x^K(\eta_s^n)\Big) f' \big( \tfrac{x}{n}\big) ds\Big)^2\Big]\\&
 \leq c_5(\rho)tK\frac{K^\alpha }{n^\alpha}  n^{2\alpha-3} \sum_x f' \big( \tfrac{x}{n} \big)^2 \sum_{y=L}^{K-1} \frac{1}{y^{2\alpha}}
	\leq \frac{C(f,\rho) t K^{1+\alpha}}{n^{2-\alpha} L^{2\alpha-1}}.
\end{split}
\end{equation}
If we take $L = n^\gamma$ and $K = n^{\gamma'}$, this last quantity goes to $0$ as soon as
\[
\gamma' < \frac{2-\alpha}{1+\alpha} + \frac{2\alpha-1}{1+\alpha} \gamma.
\]
Now, we note that if we plug this estimate into \eqref{est1} we obtain a bound on $t$ with exponent one, which again is not enough to show tightness as an application of Proposition \ref{KCC}. To solve this problem,  we take $\delta>0$ such that \[
\delta + \gamma' (1+\alpha)\leq 2-\alpha + \gamma (2\alpha-1)
\] 
and now  we get the bound $tn^{-\delta}$ for \eqref{est1}, which vanishes as $n\to\infty$ but still does suit out purposes. Nevertheless, if we think about this inequality as a recurrence, we see that it has an attractive fixed point at $\gamma=1$. Therefore, we have the following
\begin{lemma}
\label{l3.3.1}
For any $0<\delta <2\alpha$ there exists a finite sequence $\{\gamma_0,\gamma_1,\dots,\gamma_\ell\}$ such that $\gamma_0=0$, $\gamma_\ell <1-\delta$ and
\[
\frac{\delta}{1+\alpha} + \gamma_i \leq \frac{2-\alpha}{1+\alpha} + \gamma_{i-1} \frac{2\alpha-1}{1+\alpha}
\]
for any $i=1,\dots,\ell$.
\end{lemma}
The multiscale analysis of $A_t^n(f)$ goes by fixing $0<\delta<\frac{1}{2\alpha-1}$ and  defining the scales $K^i = K_n^i = n^{\gamma_i}$ for $i=1,\dots,\ell$, where the sequence $\{\gamma_0,\gamma_1,\dots,\gamma_\ell\}$ is given by Lemma \ref{l3.3.1}. Choosing $L = K_n^i$ and $K = K_n^{i+1}$ and plugging this into \eqref{est1}, we see that the expectation in  \eqref{est1} is bounded by $C(f,\rho) t n^{-\delta}$. Now, let us write $\Psi_x^i = \psi_x^{K^i}$. In this case we have the following result whose proof is given in Appendix \ref{sec_a6}.

\begin{lemma}
\label{lem_5} We have that
\begin{equation}
\label{ec3.2.2_new_new_new}
\bb E_n \Big[ \Big( \int_0^t \!\!\!n^{\alpha-3/2} \sum_x \sum_{i=1}^\ell \sum_{y = K^{i-1}}^{K^i-1} \!\!\!\!\!y a(y) \Big(\bar{\eta}_s^n(x) \bar{\eta}_s^n(x+y)-\Psi_x^i(\eta_s^n)\Big) f' \big(\tfrac{x}{n}\big) ds \Big)^2 \Big]\leq C(f,\rho) \frac{t \ell^2}{n^\delta}.
\end{equation}
\end{lemma}
Note that  the time integral in the previous expression is equal to 
\begin{equation*}\label{eq4}
\begin{split}
\int_0^t \!\!\!n^{\alpha-3/2} \sum_{x} \sum_{y=1}^{K^\ell-1} y a(y) \bar{\eta}_s^n(x) \bar{\eta}_s^n(x+y) f'\big(\tfrac{x}{n}\big) ds   - \sum_{i=1}^\ell  \int_0^t \!\!\!n^{\alpha-3/2} \sum_x M^i \Psi_x^i(\eta_s^n) f'\big(\tfrac{x}{n}\big) ds,
\end{split}
\end{equation*}
 where $M^i = \!\!\sum_{y=K^{i-1}}^{K^i-1}\!\! y a(y).$

Now we have to bound the variance of 
\begin{equation}\label{eq5}
\begin{split}
&\sum_{i=1}^\ell   \int_0^t n^{\alpha-3/2} \sum_x M^i \Psi_x^i(\eta_s^n) f'\big(\tfrac{x}{n}\big) ds-\int_0^t n^{\alpha-3/2} \sum_{x}\sum_{y=1}^{K^\ell-1} y a(y) \Psi_x^\ell(\eta_s^n) f' \big( \tfrac{x}{n}\big) ds\\
=&\int_0^t n^{\alpha-3/2} \sum_x \sum_{i=1}^\ell M^i \Psi_x^i(\eta_s^n) f'\big(\tfrac{x}{n}\big) ds - \int_0^t n^{\alpha-3/2} \sum_x \sum_{i=1}^\ell M^i\Psi^\ell_x(\eta_s^n) f'\big(\tfrac{x}{n}\big) ds\\
=&\int_0^t n^{\alpha-3/2} \sum_x \sum_{i=1}^\ell M^i\Big( \Psi_x^i(\eta_s^n) -\Psi^\ell_x(\eta_s^n)\Big) f'\big(\tfrac{x}{n}\big) ds\\
=&\int_0^t n^{\alpha-3/2} \sum_x \sum_{i=1}^\ell M^i \sum_{j=i}^{\ell-1}\Big(\Psi_x^j(\eta_s^n) -\Psi^{j+1}_x(\eta_s^n)\Big) f'\big(\tfrac{x}{n}\big) ds.
\end{split}
\end{equation}
Above we used the fact that $\sum_{y=1}^{K^\ell-1} y a(y)=\sum_{i=1}^{\ell} M^i$ which  is bounded by the mean $m = \sum_{y >0} y a(y),$ which is finite since $\alpha>1$. In order to estimate the variance of last term we use first the next result whose proof is given  Appendix \ref{sec_a8}.

\begin{lemma}\label{lem_multi_scale_1}
We have that 
\begin{equation}\label{multiscaleCPAM}
\bb{E}_n\Big[\Big(\int_0^t n^{\alpha-3/2} \sum_{x} \big( \Psi_x^j(\eta_s^n) - \Psi_x^{j+1}(\eta_s^n)\big) f' \big(\tfrac{x}{n} \big) ds\Big)^2\Big]\leq 
C(f,\rho)t\frac{(K^{j+1})^{\alpha-1}}{n^{2-\alpha}}.
\end{equation}
\end{lemma}

Now, by Fubini's Theorem we can rewrite  \eqref{eq5}  as  
\begin{equation*}
\int_0^t n^{\alpha-3/2} \sum_x \sum_{j=0}^{\ell-1}\Big(\Psi_x^j(\eta_s^n) -\Psi^{j+1}_x(\eta_s^n)\Big)\sum_{i=0}^{j} m^i  f'\big(\tfrac{x}{n}\big) ds,
\end{equation*}
and its variance is bounded by
\[
C(f,\rho)t\ell \sum_{j=0}^{\ell-1} \Big(\sum_{i=0}^j m^i\Big)^2 \frac{(K^{j+1})^{\alpha-1}}{n^{2-\alpha}}
\leq C(f,\rho) t\ell^2 m^2  n^{-(3-2\alpha) -\delta(\alpha-1)}.
\]
Last bound is obtained from \eqref{multiscaleCPAM} and with the choice of $\gamma_\ell$ given in Lemma \ref{l3.3.1}. 
Since $\alpha \leq 3/2$, the exponent of $n$ in this last expression is negative. Summarizing the estimates we have proved so far and writing $\varepsilon =\delta(\alpha-1)$, we have just showed that
\begin{lemma}
\label{l3.3.2} For any $0<\varepsilon<(2-\alpha)(\alpha-1)$ there exist $C= C(f,\rho)$ and $\ell = \ell(\varepsilon)$ such that
\[
\bb E_n \Big[ \Big( \int_0^t n^{\alpha-3/2}\!\!\!\! \sum_{|y-x| \leq K} \!\!\!\!(y-x) a(y-x) \big\{\bar{\eta}_s^n(x) \bar{\eta}_s^n(y) -\psi_x^{K_n}(\eta_s^n) \big\} f' \big(\tfrac{x}{n}\big) ds\Big)^2 \Big] \leq
	\frac{C t}{n^{\varepsilon}},
\]
where $K = n^{1-\delta}$ and $\delta=\frac{\varepsilon}{\alpha-1}$.
\end{lemma}
Finally, using the Cauchy-Schwarz inequality we see that
\[
\begin{split}
\bb E_n\Big[& \Big( \int_0^t n^{\alpha-3/2} \sum_{|y-x| \leq K} 
		(y-x) a(y-x) \psi_x^{K_n}(\eta_s^n)
		f' \big(\tfrac{x}{n} \big) ds \Big)^2 \Big] \\
&\leq c(\rho) t^2 n^{2\alpha-3} \sum_x \frac{K f'(\tfrac{x}{n})^2 m^2}{K_n^2}\leq C(f,\rho) t^2 n^{2\alpha-3+\delta}.
\end{split}
\]
It is exactly on the last line the only place where we need to assume that $\alpha <3/2$. If $\alpha <3/2$, we have just proved that for any $0<\delta<3-2\alpha$, there exists a constant $C = C(f,\rho,\delta,T)$ such that the variance of
\begin{equation}
\label{ec3.3.2}
\int_0^t n^{\alpha-3/2} \!\!\!\!\sum_{|y-x| \leq n^{1-\delta}}\!\!\!\! (y-x) a(y-x) \bar{\eta}_s^n(x) \bar{\eta}_s^n(y) f' \big( \tfrac{x}{n}\big) ds
\end{equation}
is bounded by $C t n^{-\delta}$. Recall the rough bound $C t^2 n^{2\alpha-2}$ for the variance of \eqref{ec3.3.2} obtained in Section \ref{s3.1}. By Lemma \ref{lvar} we conclude that there exist $C, \varepsilon,\delta > 0$ such that the variance of \eqref{ec3.3.2} is bounded by $Ct^{1+\delta}n^{-\varepsilon}$. By Proposition \ref{p2.2.3}, we conclude that the sequence $\{A_t^n(f); t \in [0,T]\}_{n \in \bb N}$ is tight for any $\alpha<3/2$, and moreover any limit point is identically zero.

\subsubsection{\bf{Tightness of $\{A_t^n(f); t \in [0,T]\}_{n \in \bb N}$: the case $\alpha=3/2$}}
\label{s3.4}

\quad

\quad

In the previous sections we showed, for $\alpha<3/2$, that the sequence of processes $\{A_t^n; t \in [0,T]\}_{n \in \bb N}$ is tight and that any limit point is identically zero. For $\alpha =3/2$, the limit points are given by a non-trivial function of the density of particles and in particular there is no reason to believe that they are identically zero. In this section we will show tightness of $\{A_t^n(f); t \in [0,T]\}_{n \in \bb N}$ for $\alpha =3/2$. Note that in the previous section we showed that $A_t^n(f)$ is asymptotically equivalent to
\begin{equation}
\label{ec3.4.1}
m \int_0^t \sum_x \psi_x^{K_n}(\eta_s^n) f' \big( \tfrac{x}{n} \big) ds,\footnote{We actually proved this with $m$ replaced by $\sum_{y=1}^{K_n-1} y a(y)$, but this last sum converges to $m$ as $n \to \infty$.
}
\end{equation}
where $K_n = n^{1-\delta}$ for some $\delta>0$ small enough, in the sense that the difference converges to $0$ in distribution with respect to the $J_1$-Skorohod topology, as $n\to\infty$. Therefore, it is enough to prove tightness of this process. By the equivalence of ensembles \eqref{eedes} we know that $\psi_x^{K_n}(\eta)$ is well approximated by the square of the number of particles on a box of size $K_n$ around $x$. If this box were of size $\varepsilon n$, then it would be a function of the fluctuation field $\mc Y_t^n$. Therefore, our mission now will be to go from a block of size $n^{1-\delta}$ to a block of size $\varepsilon n$. This step is what we call the {\em two-blocks estimate} as can be stated as follows. 
\begin{equation}\label{bound0}
\bb E_n \Big[ \Big( \int_0^t \sum_x \big( \psi_x^{K_n}(\eta_s^n) -\psi_x^{\varepsilon n}(\eta_s^n)\big) f'\big(\tfrac{x}{n} \big) ds \Big)^2 \Big] \leq  C(f,\rho) t \sqrt{\varepsilon}.
\end{equation}
The proof of last results is given in Appendix \ref{sec_a9} and it was  introduced in \cite{GonJar1} (see also \cite{GonJar2, GonJarSet}).
By the Cauchy-Schwarz inequality together with \eqref{eedes}, we have that
\begin{equation}\label{bound1}
\bb E_n \Big[ \Big( \int_0^t \sum_x \psi_x^{K} (\eta_s^n) f' \big( \tfrac{x}{n}\big) ds \Big)^2 \Big] \leq \frac{C(f,\rho) t^2n}{K}
\end{equation}
for any $K \in \bb N$. Choosing $K = K_n$ and $K=\varepsilon n$, from the previous estimates,  we obtain the bound
\[
\bb E_n \Big[ \Big( \int_0^t \sum_x \psi_x^{K_n}(\eta_s^n) f' \big( \tfrac{x}{n} \big) ds\Big)^2 \Big] \leq C(f,\rho) \min \{ t^2 n^\delta, t \sqrt{\varepsilon}+t^2\varepsilon^{-1}\}.
\]
Note that above, the first bound comes from \eqref{bound1} taking $K=n^{1-\delta}$, the second bound comes from \eqref{bound0}  and the last bound comes from \eqref{bound1} with $K=n\varepsilon$.
If we optimize over $\varepsilon$ in the second and third bounds, by taking $\epsilon=t^{\theta}$, we see that $\theta=2/3$ and the expectation is bounded from above by $C(f,\rho) t^{4/3}$. However, to do that we have the restriction $\varepsilon \geq n^{-\delta}$, which imposes $t \geq n^{-3\delta/2}$. For $t \leq n^{-3\delta/2}$, the first bound also gives a bound of the form $C(f,\rho) t^{4/3}$. By Proposition \ref{KCC}, we conclude that \eqref{ec3.4.1} is tight, as we wanted to show.

\section{Convergence of the sequence of density fields}

\subsection{Convergence: the case $\alpha <3/2$}
\label{s3.5}

\quad

\quad

In Sections \ref{s3.1}-\ref{s3.3} we showed that the sequence $\{\mc Y_t^n; t \in [0,T]\}_{n \in \bb N}$ is tight for $\alpha <3/2$. Let $\{\mc Y_t; t \in [0,T]\}$ be a limit point. For simplicity, up to the end of this section we denote by $n$ the subsequence along which $\{\mc Y_t^n; t \in [0,T]\}_{n \in \bb N}$ converges to $\{\mc Y_t; t \in [0,T]\}$. Recall from  \eqref{eq3.1} that
\begin{equation}
\label{ec3.5.1}
\mc Y_t^n(f) = \mc Y_0^n(f) + \int_0^t \mc Y_s^n\big(\mc L^\rho_n f \big) ds -A_t^n(f) + \mc M_t^n(f).
\end{equation} 
We want to take the limit in the previous equation  to obtain a martingale characterization of $\{\mc Y_t; t \in [0,T]\}$. 
Without loss of generality we can assume that the real-valued martingale processes $\{\mc M_t^n(f); f \in [0,T]\}_{n \in \bb N}$ converge to $\{\mc M_t(f); t \in [0,T]\}$ for any $f \in \bb S(\bb R)$, as $n\to\infty$. Note that, by the definition of the density fluctuation field given in \eqref{ec1.5.1}, the function $f$ in \eqref{ec3.5.1} is a {\em trajectory} and therefore the previous result does not apply to our setting. 

Recall from Section \ref{sec:tight_ini_field} that the initial distribution $\mc Y_0^n$ converges to a white noise of variance $\rho(1-\rho)$. In fact, for any $t \in [0,T]$ the same affirmation is true: the $\bb S'(\bb R)$-valued random variables $\mc Y_t^n$ converge in distribution to a white noise of variance $\rho(1-\rho)$. Therefore, the limit process $\{\mc Y_t; t \in [0,T]\}$ is stationary. 

Now we turn into the terms $A_t^n(f)$, $\mc M_t^n(f)$. These terms are not quite covered by the computations of Sections \ref{s3.1}-\ref{s3.3}, since the function $f$ was constant there. The martingale term is not difficult to deal with: for $t \in [0,T]$ and $\ell \in \bb N$  define $L = \lfloor t 2^n \rfloor$, $t_i = \frac{i}{2^\ell}$ and
\[
\mc M_t(f) = \lim_{\ell \to \infty} \sum_{i=0}^{L-1} \big( \mc M_{t_{i+1}}(f_{t_i}) - \mc M_{t_i}(f_{t_i})\big)
\]
Using the orthogonal increments property, we can show that $\{\mc M_t(f); t \in [0,T]\}$ is a martingale. The same approximation procedure can be done for $\{\mc M_t^n(f); t \in [0,T]\}_{n\in\bb N}$ and the limit, as $n\to\infty$, is uniform in $t$. Therefore, we conclude that
\[
\mc M_t(f) = \lim_{n \to \infty} \mc M_t^n(f)
\]
in $L^2(\bb P_n)$. This is sufficient to take the limit in \eqref{ec3.5.1} in what respects to the martingale term. The corresponding quadratic variation is equal to
\[
\<\mc M_t(f)\> = 2\rho(1-\rho)\int_0^t \mc E(f_s) ds.
\]

The computations made in Sections \ref{s3.1}-\ref{s3.3} can also be performed for smooth trajectories $f:[0,T]\to \bb S(\bb R)$. The only difference will be that the constant $C(f,\rho)$ now depends on $\rho$ and on the {\em whole trajectory} $f:[0,T] \to \bb S(\bb R)$.\footnote{
In order to avoid heavy notation, we decided to restrict the computations in previous sections to functions not depending on time.
}After this observation we conclude that $A_t^n(f)$ converges to $0$ in $L^2(\bb P_n)$, as $n\to\infty$. 

Finally, using Proposition \ref{p2} we can change $\mc L^\rho_n f_s$ by $\mc L^\rho f_s$ in \eqref{ec3.5.1}. Therefore, we are left to prove the convergence of the integral term
\[
\int_0^t \mc Y_s^n(\mc L^\rho f) ds \;\;\;\;\text{  to  } \;\;\;\; \int_0^t \mc Y_s(\mc L^\rho f) ds,
\]
as $n\to\infty$.
Recall that this last integral is defined through a limiting procedure, that is approximating $\mc L^\rho f_s$ by $\psi_\varepsilon \mc L^\rho f_s$. We can check that the approximation of $\mc Y_s^n(\mc L^\rho f_s)$ by $\mc Y_s^n(\psi_\varepsilon \mc L^\rho f_s)$ is uniform in $n$, and therefore the convergence of the integral term is guaranteed.

Putting all these elements together, we conclude that for any smooth trajectory $f: [0,T] \to \bb S(\bb R)$ we have
\[
\mc Y_t(f) = \mc Y_0(f) + \int_0^t \mc Y_s\big(  \mc L^\rho f \big) ds + \mc M_t(f),
\]
where $\{\mc M_t(f); t \in [0,T]\}$ is a continuous martingale of quadratic variation
\[
\< \mc M_t(f) \> = 2\rho(1-\rho) \int_0^t \mc E(f_s) ds.
\]
In other words, $\{\mc Y_t; t \in [0,T]\}$ is a stationary solution of \eqref{ec1.5.2}. By Proposition \ref{p1.4.3}, the distribution of $\{\mc Y_t; t \in [0,T]\}$ is uniquely determined. We conclude that the sequence $\{\mc Y_t^n; t \in [0,T]\}_{n \in \bb N}$ has a unique limit point, and therefore it actually converges to this limit point. This ends the proof of Theorem \ref{t1.5.1}.

\subsection{Convergence along subsequences: the case $\alpha=3/2$}
\label{s3.6}

\quad

\quad

In Section \ref{s3.4} we showed that the sequence of processes $\{\mc Y_t^n; t \in [0,T]\}_{n \in \bb N}$ is tight for $\alpha=3/2$. As in the previous section, let $\{\mc Y_t; t \in [0,T]\}$ be one of its limit points. For simplicity we call $n$ the subsequence over which $\{\mc Y_t^n; t \in [0,T]\}_{n \in \bb N}$ converges to $\{\mc Y_t; t \in [0,T]\}$. The treatment of the initial field, the  martingale and the  integral term in \eqref{ec2.2.1} remains the same as in the previous section. The difference between the case $\alpha <3/2$ and $\alpha =3/2$ comes from the term $A_t^n(f)$. We showed in Section \ref{s3.3} that $A_t^n(f)$ is asymptotically equivalent to\footnote{The generalization of the arguments to time-dependent test functions can be done as explained in Section \ref{s3.5}
}
\[
m \int_0^t \sum_{x } \psi_x^{K_n}(\eta_s^n) f'_s\big(\tfrac{x}{n}\big) ds.
\]
In Section \ref{s3.4} we showed in \eqref{bound0} that
\begin{equation}
\label{ec3.6.1}
\bb E_n \Big[ \Big( \int_0^t \sum_{x } \big(\psi_x^{K_n} (\eta_s^n) - \psi_x^{\varepsilon n}(\eta_s^n)\big) f'_s \big(\tfrac{x}{n} \big) ds\Big)^2 \Big] \leq C(f,\rho) t \sqrt{\varepsilon}.
\end{equation}
This bound is uniform in $n$, so if we are able to show that $$\int_0^t \sum_x \psi_x^{\varepsilon n}(\eta_s^n) f'_s \big(\tfrac{x}{n} \big) ds$$ is asymptotically equivalent to a function of the process $\{\mc Y_t^n; t \in [0,T]\}$ we will be close to prove Theorem \ref{t1.6.2}. In Section \ref{s1.6} we introduced a general approximation of the identity $\iota_\varepsilon$. In this section we use the particular choice $\iota_\varepsilon(x) = \frac{1}{\varepsilon}\mathbf{1}\{x \in (0,\varepsilon]\}$. This is specially convenient because of the identity
\[
\frac{1}{\varepsilon \sqrt n} \sum_{i=1}^{\varepsilon n} \bar{\eta}_t^n(x+i) = \mc Y_t^n \ast \iota_\varepsilon \big( \tfrac{x}{n} \big).
\]
Note that last identity is a consequence of the fact that $$\mc Y_t^n \ast \iota_\varepsilon \big( \tfrac{x}{n} \big)=\frac{1}{\sqrt n}\sum_y\iota_\varepsilon(y-x)\bar{\eta}^n_t(y).
$$
In terms of this notation the equivalence of ensembles \eqref{eedes} gives
\begin{equation}
\label{eenew}
\bb E_n \big[ \big(\psi_x^{\varepsilon n}(\eta_t^n) - \tfrac{1}{n} \Big(\mc Y_t^n \ast \iota_\varepsilon\big(\tfrac{x}{n} \big)\Big)^2 + \tfrac{\rho(1-\rho)}{\varepsilon n}\big)^2 \big] \leq \frac{c(\rho)}{(\varepsilon n )^3}.
\end{equation}
Using this bound we can see that
\begin{equation}
\label{ec3.6.2}
\bb E_n\Big[ \Big( \int_0^t \sum_x\big( \psi_x^{\varepsilon n}(\eta_s^n) - \tfrac{1}{n} \mc Y_s^n\ast \iota_\varepsilon \big(\tfrac{x}{n}\big)^2\big) f'_s\big(\tfrac{x}{n}\big) ds\Big)^2 \Big] \leq C(f,\rho,t) \big(\tfrac{1}{\varepsilon^2 n} + \tfrac{1}{\varepsilon^2 n^4}\big).
\end{equation}
We note that the previous bound follows from the following computation: first sum and subtract $\tfrac{\rho(1-\rho)}{\varepsilon n}\sum_xf'_s\big(\tfrac{x}{n}\big)$  inside the time integral, use the inequality $(x+y)^2\leq 2x^2+2y^2)$; the first error term comes from  \eqref{eenew},  and the second one comes from the approximation of the integral by the Riemann sum. Therefore, we have just proved that
\begin{equation}
\label{ec3.6.3}
\lim_{\varepsilon \to 0} \limsup_{n \to \infty} \bb E_n \Big[ \Big( A_t^n(f) - m \int_0^t \tfrac{1}{n} \sum_x \Big(\mc Y_s^n \ast \iota_\varepsilon \big( \tfrac{x}{n} \big) \Big)^2f'_s \big( \tfrac{x}{n} \big) ds \Big)^2 \Big] =0.
\end{equation}
Now we are in position to prove Theorem \ref{t1.6.2}. First we note that \eqref{ec3.6.1} implies the bound
\[
\bb E_n \Big[ \Big( \int_0^t \sum_x \big( \psi_x^{\delta n}(\eta_s^n) -\psi_x^{\varepsilon n}(\eta_s^n) \big) f_s'\big(\tfrac{x}{n} \big) ds \Big)^2\Big] \leq C(f,\rho) \varepsilon \sqrt{t}
\]
for any $\delta < \varepsilon$. Recall \eqref{Amacfield}. Passing to the limit in the previous expression, after using \eqref{ec3.6.2}, we can prove that
\[
E\big[\big(\mc A_{0,t}^\varepsilon (f) - \mc A_{0,t}^\delta(f) \big)^2 \big] \leq C(f,\rho) \varepsilon \sqrt{t}.
\]
A careful checking of the constants $C(f,\rho)$ shows that we can choose $C(f,\rho) = C_n(f,\rho)$ in such a way that
\[
\limsup_{n \to \infty} C_n(f,\rho) = c(\rho) \int_0^t \|f'_s\|^2 ds.
\]
Noticing that the process $\{\mc Y_t; t \in [0,T]\}$ is stationary, from Definition \ref{def EE}, the previous bound  shows that $\{\mc Y_t; t \in [0,T]\}$ satisfies an energy estimate with $\kappa_0 = c(\rho)$ and $\beta = \frac{1}{2}$. Therefore, the process $\{\mc A_t; t \in [0,T]\}$ given by
\[
\mc A_t(f) = \lim_{\varepsilon \to 0} \mc A_{0,s}(f)
\]
is well defined and by \eqref{ec3.6.3} we have that
\[
\lim_{n \to \infty} A_t^n(f) = m \mc A_t(f).
\]
Let us recall what we have proved about the process $\{\mc Y_t; t \in [0,T]\}$. In the previous section, we showed that $\{\mc Y_t; t \in [0,T]\}$  is stationary, and in particular it is USC, see the comments below \eqref{ec3.5.1}. We just proved that it satisfies an energy estimate for $\beta = \frac{1}{2}$. We proved that the discrete process $\{A_t^n; t \in [0,T]\}$ converges in distribution to the process $\{\mc A_t; t \in [0,T]\}$, which is well defined in virtue of the energy condition. The arguments of the previous section shows that
\[
\mc Y_t(f) - \mc Y_0(f) - \int_0^t \mc Y_s\big(\mc L^\rho f \big) ds + m \mc A_t(f)
\]
is a continuous martingale of quadratic variation
\[
2\rho(1-\rho) \int_0^t \mc E(f_s) ds.
\]
This is exactly what we called a {\em stationary energy solution} of the Burgers equation \eqref{ec1.5.3}. This ends  the proof of Theorem \ref{t1.6.2}.

\section{Discussions and remarks}

\subsection{Around the KPZ universality class}

\label{s6.1}

The stochastic Burgers equation
\begin{equation}
\label{SBE}
d \mc Y_t = \Delta \mc Y_t dt+ \lambda\nabla \mc Y_t^2 dt + \nabla d \mc W_t
\end{equation}
has received a lot of attention in recent years. In a groundbreaking work, \cite{Hai1} developed a meaningful notion of solution for this equation, and proved uniqueness of such solutions. In a very different line of research, the relation of this equation with stochastic integrable systems allows to describe in a very precise way various one-dimensional marginals of its solutions, see \cite{Cor} for a review. 
The stochastic Burgers equation and its integrated counterpart, namely, the KPZ equation, are conjectured to describe the height fluctuations of growing, one-dimensional flat interfaces, or more generally fluctuation phenomena of one-dimensional stochastic systems near a stationary, non-equilibrium state. This is known as the {\em weak KPZ universality} conjecture. The adjective weak does not indicate that this conjecture is weak; it makes reference to the fact that in order to derive \eqref{SBE} from microscopic models, one needs to introduce an extra parameter that measures the asymmetry of the system, which is tuned to converge to $0$ at a proper rate. 
These systems are believed to belong to the so-called {\em KPZ universality class}, which has scaling exponents $1:2:3$. In a more precise way, let $\varepsilon >0$ be a scaling parameter and consider the process $\{\mc Y_t^\varepsilon; t \geq 0\}$ formally defined as
\begin{equation}
\label{uva}
\mc Y_t^\varepsilon(x) = \varepsilon^{1/2} \mc Y_{t/\varepsilon^{3/2}}(x/\varepsilon).
\end{equation}
Then the process $\{\mc Y_t^\varepsilon; t \geq 0\}$ is a solution of the equation
\[
d \mc Y^\varepsilon_t = \varepsilon^{1/2} \Delta \mc Y_t^\varepsilon dt + \lambda \nabla (\mc Y_t^\varepsilon)^2 dt + \varepsilon^{1/4} \nabla d{\mc W}_t.
\]
As $\varepsilon \to 0$, this process should converge to a well-defined process, the so-called {\em KPZ fixed point}. By construction, this process has to be invariant under the $1:2:3$ scaling defined above. 
In \cite{CorQuaRem}, the authors propose a candidate for the KPZ fixed point. However, it seems that even the existence of the object defined in \cite{CorQuaRem} is not proved rigorously. It is not difficult to show that, at least formally, \eqref{ec1.5.3} is invariant under the $1:2:3$ scaling. 
In a more rigorous way, we have the following result:
\begin{theorem}
\label{papa}
The set of stationary, energy solutions of \eqref{ec1.5.3} is invariant under the space-time renormalization group of exponents $1:2:3$.
\end{theorem}
The proof of this theorem is elementary, so we omit it. Note that \eqref{ec1.5.3} is not one equation, but a two-parameter family of equations. One can parametrize \eqref{ec1.5.3} by the speed $m$ and the skewness of the operator $\mc L^\rho$; the factor $2\rho(1-\rho)$ fixes the variance of the stationary solution of the equation and can be set to be equal to $1$ by means of a spatial rescaling.
Therefore, we have derived a two-parameter family processes which are $1:2:3$-scale invariant in the sense of Theorem \ref{papa}, as scaling limits of interacting particle systems. Since we are not able to show uniqueness of solutions of \eqref{ec1.5.3}, we can not say that these processes are scale-invariant in the sense of \eqref{uva}, or even that they are different for different values of the parameters $m$, $\rho$. 
At first sight, it is even not clear whether energy solutions of \eqref{ec1.5.3} are non-trivial. The following two propositions show that these solutions are indeed non-trivial in some sense:
\begin{proposition}
\label{pfOU}
Let $\{\mc X_t;  t \geq 0\}$ be the stationary solution of 
\[
d \mc X_t = \mc L^\rho \mc X_t dt + \sqrt{ 2\rho(1-\rho) (- \mc L^{1/2})} d \mc W_t.
\]
Then the limit 
\[
\mc A_t(f) = \lim_{\epsilon \to 0} \int_0^t \mc X_s \ast \iota_\epsilon(x)^2 f'(x) dx ds
\]
defines a non-trivial process $\{\mc A_t; t \geq 0\}$. In particular, $\{\mc X_t; t \geq 0\}$ is an energy solution of \eqref{ec1.5.3} with $m=0$.
\end{proposition}
\begin{proof}
Note that in the case $m=0$, the proof of Theorem \ref{t1.6.2} can be adapted to prove convergence of the density fluctuation field to the process $\{\mc X_t; t \geq 0\}$. The computations following \eqref{ec3.6.1} show that the energy estimate \eqref{EE} holds for $\{\mc X_t; t \geq 0\}$, which shows the proposition.
\end{proof}
\begin{proposition}
\label{pafBGS}
Let $\{ \mc Y_t; t \geq 0\}$ be an energy solution of \eqref{ec1.5.3} with $\rho =1/2$. Then,
\begin{itemize}
\item[i)] The process
\[
\mc Z_t = \lim_{\epsilon \to 0} \int_0^t \mc Y_s(\iota_\epsilon(0)) ds
\]
is well defined,
\item[ii)] 
\[
\lim_{n \to \infty} \sqrt{n} \int_0^t \big( \eta_s^n(0)-1/2\big) ds = \mc Z_t
\]
in law,
\item[iii)] for any $\delta >0$ there are positive constants $c_0$, $C_0$ such that
\[
c_0 t^{4/3 -\delta} \leq E[\mc Z_t^2] \leq C_0 t^{4/3+\delta}.
\]
\end{itemize}
\end{proposition}
\begin{proof}
The first statement is a consequence of the Energy Estimate \eqref{EE}; the proof of Theorem 2.1 in \cite{GJCPAM} can be readily adapted to our situation. Starting from \eqref{ec3.6.1}, the proof of Theorem 2.5 in \cite{GJCPAM} can be adapted to prove the convergence stated in the second statement. The third statement is a consequence of this convergence combined with Theorem 2.14 in \cite{BerGonSet}.
\end{proof}
Proposition \ref{pfOU} implies that $\{\mc X_t; t \geq 0\}$ is {\em not} an energy solution of \eqref{ec1.5.3} for any $m \neq 0$, and therefore the process $\{\mc Y_t; t \geq 0\}$ is not an Ornstein-Uhlenbeck process. Proposition \ref{pafBGS} implies that the process $\{\mc Y_t; t \geq 0\}$ has a non-trivial time evolution, since otherwise the variance of $\mc Z_t$ should grow linearly in time.
Now that we have shown that the process $\{\mc Y_t; t \geq 0\}$ is non-trivial, we discuss its relation with the KPZ fixed point. 
It has been recently proved that the fractional operator $\mc L^\rho$ with $\alpha =3/2$ appears in the scaling limit of fluctuations of one-dimensional conservative systems \cite{BerGonJar, JarKomOll}. The fractional operator appears as the outcome of a {\em Dirichlet-to-Neumann} map for the solution of a degenerate Laplace equation in the half-plane. For these models, the scaling limit of energy fluctuations are given by equation \eqref{ec0.1}. In \cite{JarKomOll}, the scaling limit involves the symmetric operator $\mc L^{1/2}$, while in \cite{BerGonJar} the operator $\mc L^\rho$ has maximal skewness. According to Spohn's Fluctuating Hydrodynamics approach to anomalous heat conduction in dimension $d=1$, equation \eqref{ec0.1} should be universal, describing energy fluctuations around stationary states of {\em zero pressure}. The same approach predicts KPZ fluctuations in other regions of phase space. 
We propose the following conjecture:
\begin{conjecture}
There is at most one stationary energy solution of \eqref{ec1.5.3}. Let $\{\mc Y_t^{m,\rho}; t \geq 0\}$ denote this unique solution. Then,
\[
\lim_{m \to 0} \mc Y_t^{m,\rho} = \mc X_t^\rho,
\]
solution of \eqref{ec1.4.1}, and the limit
\[
\lim_{m \to +\infty} \mc Y_{t/m}^{m,1/2} 
\]
exists and coincides with the KPZ fixed point.
\end{conjecture}
The main issue regarding this conjecture is that we derive the fractional Burgers equation as the scaling limit of a non-local interacting particle system, in despite of the local systems which form the basin of attraction of the KPZ fixed point. In order to make this conjecture more reasonable, we explain two local mechanisms which lead to the production of the fractional terms appearing in \eqref{ec1.5.3}:
\begin{itemize}
\item
In \cite{BerGonJar, JarKomOll}, the fractional operator $\mc L^\rho$ appears as the outcome of a {\em Dirichlet-to-Neumann} map for the solution of a degenerate Laplace equation in the half-plane, which is a local, two-dimensional problem. A similar Laplace problem can be set for the asymmetric, simple exclusion process using the degree-preserving part $\mc A_0$ of the generator in the generalized duality decomposition (see (5.16) and (5.17) in \cite{KomLanOll}). This degree-preserving part has a factor $1-2\rho$ in front of it, and therefore it vanishes at $\rho=1/2$. The limit $m \to \infty$  here corresponds heuristically to take a density  $\rho \to1/2$.
 \item
The system of equations
\[
\left\{
\begin{array}{r@{=}l}
d \mc Y_t^\epsilon & \epsilon \Delta \mc Y_t^\epsilon dt + \epsilon^{-1/3} \nabla \mc Y_t^\epsilon dt + \sqrt{\epsilon} d \mc W_t,\\
d \mc Z_t^\epsilon & \nabla (\mc Y_t^\epsilon)^2
\end{array}
\right.
\]
has the striking property that solutions remain bounded in $L^2$. In fact, it can be checked that, at least heuristically, as $\epsilon \to 0$ and then $t \to \infty$ the process $\mc Z_t^\epsilon$ can be well approximated by a noise with the same law of $\sqrt{(-\mc L^{1/2})}d \mc W_t$. The exact meaning of the divergent transport term and the connection with the KPZ fixed point remains unclear.
\end{itemize}
\subsection{Weakly (a)symmetric systems}
\label{s6.2} In \cite{GubJar}, a family of fractional Burgers equations was introduced. More precisely, the concept of stationary energy solutions of
\begin{equation}
\label{fracBurgers}
d \mc Y_t = (\mc L^{1/2}) \mc Y_t dt +\lambda \nabla \mc Y_t^2 dt + \sqrt{-\mc L^{1/2}} d \mc W_t
\end{equation}
was introduced, although in finite volume. Existence was shown for $\alpha >1$ and uniqueness was shown for $\alpha >10/4$. Introducing {\em weak} (a)symmetries into the system, we can obtain these equations as scaling limits of long-range exclusion processes. More precisely, consider the family of transition rates $\{p_n(\cdot); n \in \bb N\}$ given by
\[
p_n(z) = \frac{c}{|z|^{1+\alpha}} + \lambda n^{3/2-\alpha} \mathbf{1}(z=1).
\]
For $\alpha >3/2$ this model is {\em weakly asymmetric} in the sense that the asymmetric part of the rate vanishes, as $n \to \infty$, and for $\alpha < 3/2$ this model is {\em weakly symmetric} in the sense that the asymmetric part of the transition rate grows to infinity, as $n \to \infty$. For $\alpha =3/2$, the asymmetric and symmetric parts of the transition rate are perfectly balanced. The interested reader may verify that the proof of Theorem \ref{t1.6.2} can be carried out for  this family of transition rates, and the result stated there holds for the fractional Burgers equation  \eqref{fracBurgers}.
\subsection{Normal domains of attraction}
\label{s6.3}
We say that a transition rate $p(\cdot)$ is in the {\em normal domain of attraction} of an $\alpha$-stable law if there exist constants $c^+, c^- \geq 0$ such that $c^++c^->0$ and
\[
\lim_{x \to +\infty} x^\alpha \sum_{\pm y  \geq x} p(y) = c^{\pm}.
\]
It is well known, see \cite{BorIbra} and \cite{Brei}, that the random walk with transition rate $p(\cdot)$ converges to a non-trivial Markov process under the scaling of Proposition \ref{p1.1.2} if, and only if, $p(\cdot)$ belongs to the domain of normal attraction of an $\alpha$-stable law. If the transition rate $p(\cdot)$ is symmetric, it can be checked that Theorem \ref{t1.5.1} holds for any $\alpha \in (0,2)$. Note that in this case the process $A_t^n(f)$ is identically null. In the case of non-symmetric transition rates $p(\cdot)$, the model is truly non-linear and we need the full power of Proposition \ref{p2.3.4} in order to handle $A_t^n(f)$. In order to prove Proposition \ref{p2.3.4} we need to prove the corresponding spectral gap inequality. It turns out that this is a non-trivial question, which is answered in \cite{Jar2}. With the spectral gap inequality at our disposal, we can check whether the proofs of Theorems \ref{t1.5.1} and \ref{t1.6.2} can be generalized to transition rates on the normal domain of attraction of $\alpha$-stable laws. It turns out that the proofs can be generalized without any extra assumption on the {\em symmetric} part $s(\cdot)$ of the rate $p(\cdot)$. However, some additional technical condition on the asymmetric part $a(\cdot)$ of the transition rate is needed to repeat the proof. This technical condition is not general enough in order to handle arbitrary transition rates on the normal domain of attraction of an $\alpha$-stable law, but a huge class of them. The proof becomes extremely technical without adding any insight on the models, and therefore we decided to omit it.

\subsection{Generalization to other models}
\label{s6.4}
A natural question related to the  universality property  is whether the results of these notes can be generalized to other models. The scheme of proof presented here can be applied for models in which the symmetric part of the dynamics satisfies the {\em gradient condition} with {\em local functions}. Roughly speaking, a model satisfies the gradient condition if the current of particles between two sites $x, y$ can be written as $(\tau_y -\tau_x)h$, where $h$ is a local function and $\tau_x, \tau_y$ are the shifts of  $x, y$. In the exclusion process, the current is equal to $\eta(y) -\eta(x)$, so the gradient condition is satisfied. It is very difficult to find interacting particle systems satisfying this property. In \cite{Jar1} it is observed that the zero-range process also satisfies this property, which is used to obtain the hydrodynamic limit of such a model. More examples can be constructed using the family of {\em misanthrope processes} introduced in \cite{Coc}, but even among this class of models, the gradient condition is very restrictive.
In the context of stochastically perturbed Hamiltonian dynamics it is very easy to construct models for which the techniques of these notes would allow to prove similar results. Just to give a simple example, consider the Markovian dynamics in $\bb R^{\bb Z}$ generated by $L = S +A$, where
\[
S = \sum_{x,y } s(y-x) e^{\frac{1}{2} \sum_z \eta(z)^2} \big(\partial_y -\partial_x\big) e^{-\frac{1}{2} \sum_z \eta(z)^2} \big(\partial_y-\partial_x\big),
\]
\[
A = \sum_{x,y} a(y-x) e^{\frac{1}{2} \sum_z \eta(z)^2} \big(\tau_yb\partial_y-\tau_xb\partial_x\big) e^{-\frac{1}{2} \sum_z \eta(z)^2},
\]
where $b$ is some local function. For this dynamics, the spectral gap over boxes of finite sites is well understood \cite{Cap, LanPanYau}, and product invariant measures are readily guaranteed by the construction of the dynamics.
\section*{Acknowledgements}

\noindent
P.G. thanks FCT/Portugal for support through the project 
UID/MAT/04459/2013.  
M.J. was supported by FAPERJ through the grant E-26/103.051/2012 ``Jovem Cientista do Nosso Estado''. M.J.~was partially supported by NWO Gravitation Grant 024.002.003-NETWORKS.

This project has received funding from the European Research Council (ERC) under  the European Union's Horizon 2020 research and innovative programme (grant agreement   No 715734). 

This work benefited from the support of the project EDNHS ANR-14-CE25-0011 of the French National Research Agency (ANR).

\appendix

\section{The spectral gap inequality} \label{ap:A}

A classical problem in the theory of Markov chains is the study of the time that the chain needs to reach the equilibrium. In the case of a (continuous time) finite state ergodic Markov chain it is known that the convergence to equilibrium is exponentially fast. Therefore the relevant question is the exponential rate at which this happens. Let $\{x(t); t \geq 0\}$ be an ergodic Markov chain on a finite state space $E$. Let $\mu$ be its unique invariant measure. For $f: E \to \bb R$ and $x \in E$, let $P_t f(x) = \bb E [ f(x(t))| x(0)=x]$. Let $\<\cdot\>_\mu$ denote the expectation with respect to $\mu$. Then we define
\[
\lambda = -\!\! \sup_{f:E \to \bb R} \limsup_{t \to \infty} \frac{1}{t} \log (\big\| P_t f- \<f\>_\mu\big\|_{L^2(\mu)}).
\]
The number $1/\lambda$ is known as the {\em relaxation time} of the chain $\{x(t); t \geq 0\}$. In the case on which the chain $\{x(t); t \geq 0\}$ is reversible with respect to $\mu$, the number $\lambda$ is equal to the {\em spectral gap} of the generator $A$ of the chain $\{x(t); t \geq 0\}$, that is, the absolute value of the largest non-zero eigenvalue of $A$. In that case, we have the variational formula
\begin{equation}
\label{SG}
\lambda^{-1} = \sup_{\<f\>_\mu=0} \frac{\<f^2\>_\mu}{\<-fAf\>_{\mu}}.
\end{equation}
When the chain is not reversible, this variational formula provides an upper bound for $\lambda^{-1}$.  A natural question in the theory of Markov chains is to estimate the spectral gap of a Markov chain, or of a family of Markov chains of increasing complexity.

For the symmetric simple random walk restricted to $\{1,\dots,n\}$ it is well known that $\lambda^{-1} = \mc O(n^2)$. It turns out that this property of the random walk over finite intervals, by means of a computation of Nash type, allows one to show that in the case of the symmetric simple random walk on $\bb Z$,
\[
\big\| P_t f - \< f\>_\mu\big\|_{L^2(\mu)} = o\big(\tfrac{1}{t^a}\big) \text{ for any } a < \tfrac{1}{2},
\]
where $\mu$ is the counting measure on $\bb Z$.
Therefore a sharp upper bound on the spectral gap of finite-state Markov chains gives valuable information even in the case of chains on infinite state spaces.

Consider the symmetric random walk restricted to the set $\Lambda_\ell = \{1,\dots,\ell\}$. In the case where the random walks have long jumps we have the following result:

\begin{proposition}
\label{p1.3.1}
Let $p(\cdot)$ be given by \eqref{ec1}. There exists $\kappa>0$ such that
\begin{equation}
\sum_{x \in \Lambda_\ell} f(x)^2 \leq \kappa \ell^\alpha \sum_{x,y \in \Lambda_\ell} p(y-x) \big(f(y)-f(x)\big)^2
\end{equation}
for any $\ell \in \bb N$ and any $f: \Lambda_\ell \to \bb R$ such that
\begin{equation}\label{eq1}
\sum_{x \in \Lambda_\ell} f(x) =0.
\end{equation}
\end{proposition}

\begin{remark}
This proposition is telling us that the spectral gap of a Markov chain with jump rates given in \eqref{ec1} and defined on the  interval $\Lambda_\ell$, is bounded from below by $\frac{1}{\kappa \ell^\alpha}$. In addition, pairing together the two terms involving $x$ and $y$, we see that only the behavior of the symmetric part of $p(\cdot)$ is relevant for this proposition.
\end{remark}

The proof of this proposition is in fact very simple.  For that purpose note that \begin{equation}
\sum_{x,y \in \Lambda_\ell} p(y-x) \big(f(y)-f(x)\big)^2=\sum_{x,y \in \Lambda_\ell}2s(y-x)\big( f(y) -f(x) \big)^2.
\end{equation}
To conclude, use the fact that
 for $f$ satisfying \eqref{eq1}, it holds:
\[
\sum_{x,y \in \Lambda_\ell} \big( f(y) -f(x) \big)^2 = 2 \ell \sum_{x \in \Lambda_\ell} f(x)^2,
\]
together with $s(y-x) \geq \frac{c^++c^-}{2 \ell^{1+\alpha}}$ for any $x,y \in \Lambda_\ell$.

As a corollary of Proposition \ref{p1.3.1} we can obtain a lower bound for the spectral gap of the exclusion process with transition rate $p(\cdot)$:

\begin{corollary}
\label{c1.3.2}
Let $p(\cdot)$ be defined by \eqref{ec1}. Let $f: \Omega \to \bb R$ be a local function with $\supp(f) \subseteq \Lambda_\ell$. Assume that $\int f d\mu_\sigma = 0$ for any $\sigma \in [0,1]$. Then,
\[
\int f^2 d\mu_\sigma \leq \kappa \ell^\alpha \sum_{x,y \in \Lambda_\ell} p(y-x) \int \big(\nabla_{x,y} f\big)^2 d\mu_\sigma
\]
for any $\sigma \in [0,1]$.
\end{corollary}

The simplest way to prove this corollary is by means of the {\em Aldous' conjecture}, proved in \cite{CapLigRit}, which says that the spectral gap of an exclusion process with {\em symmetric} rates is equal to the spectral gap of the random walk with the same rates. Another proof using a comparison principle can be found in \cite{Jar2}.

\section{Proof of Proposition \ref{p2}}\label{proofprogen}

We do the proof for the case $\alpha>1$, the others being analogous.
The proof of the proposition is elementary, but very tedious. Recall the definition of $\mathcal{L}_n$ and $\mathcal{L}$ from \eqref{ec'2} and \eqref{ec2}, respectively. First note that by the definition of $m_n^\alpha$ in \eqref{mnalpha}, for this regime of $\alpha$
we rewrite 
$\mathcal{L}_n f(\tfrac{x}{n}) = n^\alpha \sum_{y} p(y)\psi_{x/n}(\tfrac{y}{n})$ and $\mc{L} f(x) = \int_\bb R p(y) \psi_x(y)dy
$
where $
\psi_{u}(v) = f(u+v)-f(u)-vf'(u),$ for $f\in \mc C^2(\bb R)$.
Second, note that
$\mc{L}_n^+ f(\tfrac{x}{n}) = n^\alpha \sum_{y \geq 1} p(y)\psi_{x/n}(\tfrac{y}{n})$ and 
$
\mc{L}^+ f(x) = \int_0^\infty \frac{c^+}{y^{1+\alpha}} \psi_x(y)dy
$
are well-defined and that it is enough to show \eqref{resprop2.2} for $\mc{L}_n^+$ and $\mc{L}^+$. For $x \in \bb N$, define $P(x) = \sum_{y \geq x} p(y)$ and $a(x) = x^\alpha P(x)-\frac{c^+}{\alpha}$. Note that $a(x)$ tends to $0$, as $x\to\infty$. The idea is to perform an integration by parts in the formula of $\mc L_n^+f$ in order to work with the more regular object $P(\cdot)$. By writing $p(y)=P(y)-P(y+1)$, performing a summation by parts and a Taylor expansion on $\psi$, we see that $\mc L_n^+f(\tfrac{x}{n}) = n^{\alpha-1} \sum_{y\geq 1} P(y)\psi_{x/n}'(\tfrac{y}{n})+R_1^n(x)$,
where $R_1^n(x)$ is an error term which satisfies $|R_1^n(x)| \leq \|\psi''_{x/n}\|_\infty n^{\alpha-2} \sum_{y \geq 1} P(y).$
Note that $\|\psi_{x/n}'\|_\infty {\leq 2\|f'\|_\infty}$, which does not depend on $x$.
This last sum is equal to $\sum_{y \geq1} y p(y)<+\infty$ and since $\alpha<2$, $R^1_n(x)$ vanishes, as $n\to\infty$. For $y\geq 1$, let $A(y) = \sup_{z \geq y} |a(z)|$.
We have that 
\begin{equation}\label{useeq}
\mc L^+_nf(\tfrac{x}{n})= n^{\alpha-1} \sum_{y\geq 1} \frac{c^+}{\alpha y^\alpha}\psi_{x/n}'(\tfrac{y}{n})+n^{\alpha-1} \sum_{y\geq 1} \frac{a(y)}{y^\alpha}\psi_{x/n}'(\tfrac{y}{n}).
\end{equation}
Note that $\psi_u'$ is bounded and that $\psi_u(0)=\psi'_u(0)=0$. Therefore, there exists a constant $K$ such that $|\psi_u'(v)| \leq Kv$ for any $v\in [0,1]$ and such that $|\psi_u'(v)| \leq K$ for any $v >1$. In fact, we can choose $K=\max\{2\|f'\|_\infty, \|f''\|_\infty\}$. Therefore, the second term on the right hand side of \eqref{useeq} is bounded by
\begin{equation*}
\begin{split}
&\Big|n^{\alpha-1}\sum_{y=1}^n \frac{a(y)}{y^\alpha}\psi_{x/n}'(\tfrac{y}{n})\Big| +\Big|n^{\alpha-1}\sum_{y\geq n+1}\frac{a(y)}{y^\alpha}\psi_{x/n}'(\tfrac{y}{n})\Big| \\
	&\leq Kn^{\alpha-2} \Big(A(1)\sum_{y=1}^k y^{1-\alpha} +  A(k+1)\sum_{y=k+1}^n y^{1-\alpha}\Big)+K n^{\alpha-1} A(n)\sum_{y\geq n+1}  \frac{1}{y^\alpha}\\
&\leq \frac{K}{2-\alpha}\Big(A(1) (\tfrac{k}{n})^{2-\alpha}+A(k+1)\Big)+\frac{K A(n+1)}{\alpha-1}
\end{split}
\end{equation*}
for any $k<n$. Choosing, for example, $k = \sqrt n$ , the last sums vanish,  as $n \to \infty$.  Note that
\[
\mc L^+ f\big(\tfrac{x}{n}\big) = -\frac{c^+ \psi_x(y)}{\alpha y^\alpha}\Big|_{y=0}^\infty + \int_0^\infty \frac{c^+ \psi_x'(y)}{y^\alpha} dy = c^+\int_0^\infty \frac{ \psi_x'(y)}{\alpha y^\alpha} dy,
\]
since $\psi_x(y)$ is quadratic around $y =0$ and linear for $y \gg 1$.
Moreover, the first sum on the right hand side of \eqref{useeq} is just a Riemann sum for this last integral. Since the function $\frac{\psi_x'(y)}{y^\alpha}$ is continuous at $y=0$ and it decays like $\frac{1}{y^\alpha}$, this Riemann sum converges to the corresponding integral. Note that this convegence is uniform in $x$, since $\frac{\psi_x'(y)}{y^\alpha}$ is equicontinuous in $x$. Finally, since for $0<y<z$, we have that
$\Big|\frac{\psi'(z)}{z^\alpha}-\frac{\psi'(y)}{y^\alpha}\Big| 
		\leq \frac{C(z-y)}{y^\alpha} $
for some constant $C$ which depends only on $\|f''\|_\infty$, we  conclude that
\[
\Big| n^{\alpha-1} \sum_{y \geq 1} \frac{c^+}{\alpha y^\alpha} \psi_{x/n}'(\tfrac{y}{n}) - \int_{\tfrac{1}{n}}^\infty \frac{c^+ \psi'(y)}{y^\alpha}dy\Big|
		\leq C n^{\alpha-2} \sum_{y\geq 1} \frac{1}{y^\alpha}.
\]
Since the last sum is finite and $\alpha<2$, we have just shown that 
\[
\lim_{n \to \infty} | \mc L^+_nf(\tfrac{x}{n}) - \mc  L^+f(\tfrac{x}{n})| =0.
\]
Moreover, since all the constants above do not depend on $x$, the limit is uniform in $x$, showing the first half of the proposition. The second half can be proved in a similar way.

\section{Proof of Proposition \ref{p1.4.3}}\label{sec uniq ou}
For the reader's convenience, we repeat here various definitions introduced in Section \ref{s1}.
Let $\mc L$ be a generator of a L\'evy process in $\bb R$. Let $\{\mc W_t; t \geq 0\}$ be a Brownian motion on $L^2(\bb R)$ and let $\mc S = \frac{1}{2}(\mc L + \mc L^*)$ be the symmetric part of the operator $\mc L$. We say that a stochastic process $\{\mc Y_t; t \geq 0\}$ is a {\em stationary solution} of the infinite-dimensional Ornstein-Uhlenbeck equation
\begin{equation}
\label{ec1.A1}
d \mc Y_t = \mc L^* \mc Y_t dt + \sqrt{-2 \chi} \mc S d \mc W_t
\end{equation}
if for each $t \in [0,T]$ the random variable $\mc Y_t$ is a white noise of variance $\chi$ and for any differentiable function $f: [0,T] \to \bb S(\bb R)$ the process
\[
\mc Y_t(f_t) - \mc Y_0(f_0) - \int_0^t \mc Y_s( (\partial_s+\mc L) f_s) ds 
\]
is a martingale of quadratic variation
$2 \chi \int_0^t \< f_s , -\mc S f_s \> ds.$
We will prove following result:

\begin{proposition}
Two stationary solutions of \eqref{ec1.A1} have the same  law.
\end{proposition}
\begin{proof}
Let $f$ be a function in $\bb S(\bb R)$ and take $t \geq 0$. Let $\{P_t; t \geq 0\}$ the semigroup associated to the generator $\mc L$, that is, $P_t = e^{t \mc L}$ for any $t \geq 0$. Since $\mc L$ is the generator of  a L\'evy process, $\{P_t; t \geq 0\}$ is a strongly continuous, contraction semigroup on $\mc C_b(\bb R)$. In particular $f_s = P_{t-s} t$ is a differentiable trajectory on $\mc C_b(\bb R)$ satisfying $\frac{d}{ds} f_s = - \mc L f_s$ for any $s \leq t$ and $f_t = f$. Since $\{P_t; t \geq 0\}$ is also  a contraction in $L^1(\bb R)$, it is a contraction in $L^2(\bb R)$. Note that $\{f_s; s \leq t\}$ is not a legitimate test function, since although $P_{t-s} f$ is infinitely differentiable, it does not satisfy the decay properties needed to be in $\bb S(\bb R)$. However, $\{f_s; s \leq t\}$ can be approximated in $L^2$ by differentiable functions $f_\epsilon: [0,t] \to \bb S(\bb R)$, justifying the use of $\{f_s; s \leq t\}$ as a test function. Since $(\partial_s + \mc L) f_s =0$, we conclude that
\[
\mc M_{s,t}(f) =: \mc Y_s(P_{t-s} f) - \mc Y_0(P_t f) 
\]
is a martingale of quadratic variation (with respect to $s$)
\[
2 \chi \int_0^t\<P_s f, -\mc S P_s f\> ds.
\]
Since the quadratic variation of $\mc M_{s,t}(f)$ is deterministic, $\{\mc M_{s,t}(f);  s\in[0,  t]\}$ and in consequence $\{\mc Y_t;  t\in [0, T]\}$ are Gaussian processes.
Note that
\[
\tfrac{d}{dt} \< P_t f, P_t f\> = 2 \<P_t f, \mc L P_t f\> = -2 \<P_t f, -\mc S P_t f\>.
\]
Therefore
\[
2 \chi \int_0^t \<P_s f, -\mc S P_s f\> = \chi\big(\<f,f\> - \<P_t f, P_t f\>\big).
\]
We conclude that $\mc Y_t(f)$ can be written as the sum of two independent Gaussian variables: $\mc Y_0(P_t f)$, which depends only on the initial law and $\mc M_{t,t}(f)$, which is independent of $\mc Y_0$. Since $\{\mc Y_t;  t\in [0, T]\}$ is a Gaussian process, it is characterized by its covariance structure. By stationarity, the computations above show that for any $0 \leq s \leq  t \leq T$ and any $f,g \in \bb S(\bb R)$,
\begin{align*}
 E[ \mc Y_t(f) \mc Y_s(g)] 
 		&=  E[ \mc Y_{t-s}(f) \mc Y_0(g)] = E[(\mc Y_0(P_{t-s}f)+\mc M_{t-s,t-s}(f)) \mc Y_0(g)]
 		\\&= \chi \<P_{t-s} f,g\>,
\end{align*}
which shows uniqueness in  law of the process $\{\mc Y_t;  t\in [0, T]\}$.
\end{proof}

\section{Auxiliary computations}~\label{aux_computations}

In this section we collect the proofs of some estimates that are needed in Section  \ref{s3}. 

\subsection{Bounds on Taylor expansions of test functions}

In the computations below, it will be useful to control the decay at infinity of the error terms of the Taylor expansion for test functions. For $g:\bb R \to \bb R$, $x \in \bb R$ and $M >0$, define
\[
\|g(x)\|_{M,\infty} = \sup_{|y-x| \leq M} | g(y) |.
\]
We have the following:
\begin{lemma}
\label{palta}
Let $g \in \bb S(\bb R)$. Then, for any $\ell \in \bb N$,
\[
\lim_{x \to \pm \infty} |x|^\ell \|g(x)\|_{M,\infty} =0.
\]
\end{lemma}
The proof of this lemma is elementary, so we omit it. We can apply this lemma to obtain the aforementioned bounds on Taylor expansions of test functions: let $f \in \bb S(\bb R)$, $x \in \bb R$ and $\ell \in \bb N$. For any $M >0$ and any $y \in \bb R$ such that $|y-x| \leq M$, we have that
\begin{equation}\label{useful_bound}
\Big| f(y) - \sum_{i=0}^{\ell-1} \frac{f^{(i)}(x)}{i !} (y-x)^\ell\Big| \leq \frac{|y-x|^\ell}{\ell !} \|f^{(\ell)}(x)\|_{M,\infty}.
\end{equation}
This is an immediate consequence of the Taylor formula with Lagrange error and Lemma \ref{palta}.

\subsection{Estimation of \eqref{imp_terms}}
\label{sec_a1}
In order to estimate the first term of \eqref{imp_terms} we first make a change of variables $z=y-x$, we fix $M$ and write it as 
\begin{equation}\label{int1}
\begin{split}
&c_1(\rho) n^{2\alpha -2} \sum_{x\in \bb Z}\sum_{|z|=1}^{Mn} \frac{1}{|z|^{2+2\alpha}} \big(f\big(\tfrac{z+x}{n}\big) - f\big(\tfrac{x}{n}\big) \big)^4\\
&+c_1(\rho) n^{2\alpha -2} \sum_{x\in \bb Z}\sum_{|z|\geq M n} \frac{1}{|z|^{2+2\alpha}} \big(f\big(\tfrac{z+x}{n}\big) - f\big(\tfrac{x}{n}\big) \big)^4.
	\end{split}
\end{equation}
  The second term  in \eqref{int1} can be bounded from above by
  $$C n^{2\alpha -2} \sum_{x\in \bb Z}\sum_{|z|\geq M n} \frac{1}{|z|^{2+2\alpha}} \big(f\big(\tfrac{z+x}{n}\big)^4+C n^{2\alpha -2} \sum_{x\in \bb Z}\sum_{|z|\geq M n} \frac{1}{|z|^{2+2\alpha}} \big(f\big(\tfrac{x}{n}\big) \big)^4.$$
  By making a change of variables, using the fact that $f\in{ \bb S}(\bb R)$ and since 
  $$ n^{1+2\alpha}\sum_{|z|> Mn}\frac{1}{|z|^{2+2\alpha}} <\infty$$  
  we conclude that the two previous sums are of order $O(n^{-2})$. To compute the first term in \eqref{int1}, we use Lemma \ref{palta}   to conclude that
 \begin{equation*}
 n^{2\alpha -2} \sum_{x \in \bb Z}\sum_{|z|=1}^{Mn}
  \frac{1}{|z|^{2+2\alpha}} \big(f\big(\tfrac{z+x}{n}\big) - f\big(\tfrac{x}{n}\big) \big)^4\leq n^{2\alpha-6}\sum_{x \in \bb Z}\|f'(x/n)\|_{M/n,\infty}^4  \sum_{|z|=1}^{Mn}|z|^{2-2\alpha}.
 \end{equation*}
 Now we use \eqref{useful_bound} and since 
 $$ n^{2\alpha-3}\sum_{|z|=1}^{Mn}\frac{1}{|z|^{2\alpha-2}}=O(n^{2\alpha-3}),$$ we obtain that the second term   
 in \eqref{int1} is of order $O(n^{2\alpha-5})$ and  therefore vanishes as $n\to\infty.$ 
The second term of \eqref{imp_terms}, by a change of variables $z=y-x$ and fixing $M$, can be written  as
\begin{equation*}
\begin{split}
&c_2(\rho) \frac{1}{n} \sum_{x\in\bb Z} \Big(\frac{1}{n}\sum_{|z|=1}^{Mn} \frac{n^{1+\alpha}}{|z|^{1+\alpha}}\big(f\big(\tfrac{z+x}{n}\big)-f\big(\tfrac{x}{n}\big)\big)^2\Big)^2\\+&
c_2(\rho) \frac{1}{n} \sum_{x\in\bb Z} \Big(\frac{1}{n}\sum_{|z|\geq Mn} \frac{n^{1+\alpha}}{|z|^{1+\alpha}}\big(f\big(\tfrac{z+x}{n}\big)-f\big(\tfrac{x}{n}\big)\big)^2\Big)^2.
	\end{split}
\end{equation*}
By repeating exactly the same computations as above, the first term of last expression is of order $O(n^{\alpha-3})$ while the second one is of order  $O(n^{-1})$. This shows that the  second term of \eqref{imp_terms} vanishes, as $n\to\infty$.
\subsection{Estimation of \eqref{int:field_a}}
\label{sec_a2} The term \eqref{int:field_a}, by a change of variables $z=y-x$ and fixing $M$, can be written as
\begin{equation*}
\begin{split}
 &c_3(\rho) \frac{1}{n^2} \sum_{x\in \bb Z}\sum_{|z|=1}^{Mn} \frac{n^{2+2\alpha}}{|z|^{2+2\alpha}} \big(f\big(\tfrac{z+x}{n}\big) -f \big( \tfrac{x}{n}\big)\big)^2\\+& c_3(\rho) \frac{1}{n^2} \sum_{x\in\bb Z} \sum_{|z|\geq Mn}\frac{n^{2+2\alpha}}{|z|^{2+2\alpha}} \big(f\big(\tfrac{z+x}{n}\big) -f \big( \tfrac{x}{n}\big)\big)^2.
 \end{split}
\end{equation*}
Repeating the same procedure as above we can see that the second term above is of order 
$O(n^{-2})$, while the first one  can be bounded from above by 
\[
Cn^{2\alpha-2}\sum_{|z|=1}^{Mn}{|z|^{-2\alpha}}.
\]
Last sum is convergent if $\alpha <1/2$ and
if $\alpha \geq 1/2$, it is divergent. When $\alpha=1/2$, the order of divergence is $\log(n)$. Therefore, we can conclude that  \eqref{int:field_a}  is of order $o(1)$ for $\alpha <1$ and bounded for $\alpha =1$.

\subsection{Proof of Lemma \ref{lem_1}}
\label{sec_a3}
The expectation in \eqref{int_3} is bounded from above by 
\begin{equation}\label{int:field_a_new}
\begin{split}
 c_4(\rho) n^{2\alpha -1} \sum_{|x-y|\geq K_n} a(y-x)^2 \big(f\big(\tfrac{y}{n}\big) -f \big( \tfrac{x}{n}\big)\big)^2,\end{split}
\end{equation}
and, by the change of variables $z=y-x$ and by fixing $M$, we can write it as 
\begin{equation*}
\begin{split}
 &c_4(\rho) n^{2\alpha -1} \sum_{x\in\bb Z}\sum_{K_n\leq |z|\leq Mn} \frac{1}{|z|^{2+2\alpha}} \big(f\big(\tfrac{z+x}{n}\big) -f \big( \tfrac{x}{n}\big)\big)^2\\+& c_4(\rho) n^{2\alpha -1} \sum_{x\in\bb Z}\sum_{|z|\geq Mn}\frac{1}{|z|^{2+2\alpha}}\big(f\big(\tfrac{z+x}{n}\big) -f \big( \tfrac{x}{n}\big)\big)^2.
 \end{split}
\end{equation*}
Now, doing the same arguments as above, we see that the first term in last expression if of order $O(n^{-1})$, while the second one can be bounded from above by 
\begin{equation*}
\begin{split}
C n^{2\alpha -2}\sum_{ |z|=K_n}^{Mn}|z|^{-2\alpha}
 \end{split}
\end{equation*}
which is  of order $\frac{n^{2\alpha-2}}{K_n^{2\alpha-1}}$ and vanishes if $K_n \gg n^{\frac{2\alpha-2}{2\alpha-1}}$. 

\subsection{Proof of Lemma \ref{lem_2}}\label{sec_a4}
Note that the expectation in \eqref{int_4} is bounded from above by 
\begin{equation*}
c_4(\rho) n^{2\alpha-1} \!\!\!\!\sum_{|y-x| \leq K}\!\!\!\! a^2(y-x)  \frac{(y-x)^4}{n^4} \|f''((y-x)/n)\|^2_{K/n,\infty}.
\end{equation*}
As above, by making the change of variables $z=y-x$ and using the fact that $f\in\bb S(\bb R)$ we can bound the previous expression by
\begin{equation*}
Cn^{2\alpha-4}\sum_{|z| \leq K}z^{2-2\alpha}
\end{equation*}
which is of order $\frac{K^{3-2\alpha}}{n^{4-2\alpha}}.$
Since $K \gg n^{\frac{2\alpha-2}{2\alpha-1}}$, the expectation vanishes as $n\to\infty$.

\subsection{Proofs of Lemmas \ref{lem_3} and \ref{lem_4}}\label{sec_a5}
We start with the proof of Lemma \ref{lem_3}. By the Cauchy-Schwarz inequality and by \eqref{eedes}, the expectation in \eqref{ec3.2.3} is bounded from above by
\begin{equation}\label{bound_1}
Kt^2 n^{2\alpha-3} \Big( \sum_{y=1}^{K-1} y a(y)\Big)^2\sum_x f'\big(\tfrac{x}{n}\big)^2 \int \psi_x^{K} (\eta)^2 \mu_\rho(d\eta) \leq C(f,\rho)\frac{ t^2 n^{2\alpha-2}}{K},
\end{equation}
which vanishes, as $n\to\infty$, if $K \gg n^{2\alpha-2}$.  Above we used the fact that $\sum_{y=1}^{K-1} y a(y)<\infty$, since we are in the regime $\alpha>1$. 
Now we prove Lemma \ref{lem_4}. For each $j=1,\cdots, K$, we use
  Proposition \ref{p2.3.4} with $F(\eta)=\sum_{x \in \mc Z_j}F_x(\eta)$ and $$F_x(\eta):=n^{\alpha -3/2}  \sum_{y=1}^{K-1}y a(y) \big\{ \bar{\eta}(x) \bar{\eta}(x+y) - \psi_x^{K} (\eta)\big\} f'\big(\tfrac{x}{n} \big).$$ Therefore, the expectation in \eqref{ec3.2.3} is bounded from above by
\begin{equation}\label{bound_2}
c_4(\rho) t K \frac{K^{\alpha}}{n^\alpha} n^{2\alpha-3} \sum_{x} f'\big( \tfrac{x}{n}\big)^2 \sum_{y=1}^{K-1} y^2 a(y)^2 \leq C(f,\rho) \frac{ K^{1+\alpha} t}{n^{2-\alpha}}.
\end{equation}
We note that above we  used the fact that $\sum_{y=1}^{K-1} y^2 a(y)^2<\infty$. We conclude that last term vanishes if $K \ll n^{\frac{2-\alpha}{1+\alpha}}$.

\subsection{Proof of Lemma \ref{lem_5}}\label{sec_a6}

By the Minkowski's inequality and the  estimate in \eqref{est1}, the expectation in \eqref{ec3.2.2_new_new_new}
is bounded from above  by
\[
tC(f,\rho)\Big(\sum_{i=1}^\ell\sqrt{\frac{n^{\gamma_i (1+\alpha)}}{n^{\gamma_{i-1}(2\alpha-1)+2-\alpha}}}\Big)^2=C(f,\rho) \frac{t \ell^2}{n^\delta}.
\]

\subsection{Proof of Lemma \ref{lem_multi_scale_1}}\label{sec_a8}

 At at a first glance we note that
by Proposition \ref{p2.3.4} and by \eqref{eedes}, the expectation in  \eqref{multiscaleCPAM}
is bounded by
\begin{equation}\label{renormstep}
C(\rho,f) t n^{2\alpha-2} \frac{(K^{j+1})^{1+\alpha}}{n^\alpha}  \int (\Psi_x^j)^2 d\mu_\rho \leq C(f,\rho) t\frac{ (K^{j+1})^{\alpha+1}}{(K^j)^2n^{2-\alpha}}.
\end{equation}
But this bound will not be sufficient for us. Therefore,  in order to prove the lemma we use the multiscale structure introduced in \cite{GJCPAM}. For that purpose, 
let $j$ be fixed and take the sequence of boxes $\ell_0=K^j$,  $\ell_1=2\ell_0$ and for $p\geq 2$, $\ell_p=2^p\ell_0$. Suppose that there exists  $P$ sufficiently big such that $2^P \ell_0=K^{j+1}$. 
Performing a telescopic  sum and using  the Minkowski's inequality together with \eqref{renormstep}, the expectation on the left hand side of \eqref{multiscaleCPAM} is bounded from above by 
\begin{equation}\label{renorm_2}
C(f,\rho)t\Big(\sum_{p=1}^P{\frac{(2^{p+1}\ell_0)^{\alpha-1}}{n^{2-\alpha}(2^p\ell_0)^{2}}}\Big)^2\leq C(f,\rho)t \frac{(2^P\ell_0)^{\alpha-1}}{n^{2-\alpha}}.
\end{equation}
This proves \eqref{multiscaleCPAM} for the case $K^{j+1}=2^P \ell_0$. For the other cases, we take $P$ sufficiently big such that $2^P\ell_0\leq K^{j+1}\leq 2^{P+1}\ell_0$. Let $\tilde{\Psi}_x^i=\Psi_x^{2^i\ell_0}$. 
Then, by using the inequality $(x+y)^2\leq 2x^2+y^2$, the expectation on the left hand side of \eqref{multiscaleCPAM} is bounded from above by
\begin{equation}\begin{split}\label{multiscaleCPAM2}
&2\bb{E}_n\Big[\Big(\int_0^t n^{\alpha-3/2} \sum_{x} \big( \Psi_x^j(\eta_s^n) - \tilde{\Psi}_x^{j}(\eta_s^n)\big) f' \big(\tfrac{x}{n} \big) ds\Big)^2\Big]\\
+&2\bb{E}_n\Big[\Big(\int_0^t n^{\alpha-3/2} \sum_{x} \big( \tilde{\Psi}_x^{j}(\eta_s^n)- {\Psi}_x^{j+1}(\eta_s^n)\big) f' \big(\tfrac{x}{n} \big) ds\Big)^2\Big].
\end{split}
\end{equation}
From \eqref{renorm_2}, the first expectation in \eqref{multiscaleCPAM2} is bounded from above by $$C(f,\rho)t \frac{(2^P\ell_0)^{\alpha-1}}{n^{2-\alpha}},$$ while from \eqref{renormstep} the second expectation is bounded from above by
\begin{equation*}
C(f,\rho) t\frac{ (K^{j+1})^{\alpha+1}}{(2^P\ell_0)^2n^{2-\alpha}}\leq C(f,\rho) t\frac{ (2^{P+1}\ell_0)^{\alpha+1}}{(2^P\ell_0)^2n^{2-\alpha}}.
\end{equation*}
Putting together the two previous estimates, the proof of \eqref{multiscaleCPAM} ends.

\subsection{Proof of \eqref{bound0}}\label{sec_a9}
First we compute the price to double the size the box. For that purpose, let $M$ be given. By Proposition \ref{p2.3.4} and \eqref{eedes} we have that
\[
\bb E_n \Big[ \Big( \int_0^t \sum_x \big( \psi_x^M(\eta_s^n) -\psi_x^{2M}(\eta_s^n)\big) f'\big(\tfrac{x}{n} \big) ds \Big)^2 \Big] \leq C(f,\rho) t \sqrt{\frac{M}{n}}.
\]
Define $M_0=M$ and $M_i = 2^i M$ for $i \in \bb N$. By writing a telescopic sum, using Minkowski's inequality and the previous estimate, we see that
\[
\begin{split}
&\bb E_n \Big[ \Big( \int_0^t \sum_x \big( \psi_x^M(\eta_s^n) -\psi_x^{M_\ell}(\eta_s^n)\big) f'\big(\tfrac{x}{n} \big) ds \Big)^2 \Big]\\
=&\bb E_n \Big[ \Big( \int_0^t \sum_x \sum_{i=0}{\ell-1}\big( \psi_x^{2^iM}(\eta_s^n) -\psi_x^{2^{i+1}M}(\eta_s^n)\big) f'\big(\tfrac{x}{n} \big) ds \Big)^2 \Big]\\
	\leq& \frac{C(f,\rho) t}{\sqrt n} \Big(\sum_{i=0}^{\ell -1} M_i^{1/4}\Big)^2\leq C(f,\rho)t \sqrt{\frac{M_\ell}{n}}.
\end{split}\]
Taking $M= K_n$ and $M_\ell = \varepsilon n$ the proof ends.

\end{document}